\documentclass[reqno]{amsart}
\usepackage{amsfonts,amsmath,amsthm,amssymb,mathrsfs}
\usepackage{color}
\usepackage{amsmath,amssymb,amsfonts,bm}
\usepackage{graphicx}
\usepackage{CJKutf8}
\usepackage{newunicodechar}
\usepackage{xcolor}
\numberwithin{equation}{section}
\usepackage[pdfstartview=FitH,colorlinks=true]{hyperref}
\hypersetup{urlcolor=red, linkcolor=blue,citecolor=red, CJKbookmarks=true}
\allowdisplaybreaks

\theoremstyle{plain}
\newtheorem{theorem}{Theorem}[section]
\newtheorem{remark}[theorem]{Remark}
\newtheorem{lemma}[theorem]{Lemma}
\newtheorem{proposition}[theorem]{Proposition}
\newtheorem{definition}[theorem]{Definition}
\newtheorem{corollary}[theorem]{Corollary}
\newtheorem{Example}{Example}[section]
\numberwithin{equation}{section}

\begin{document}

\title[Landau equation]
{Cauchy problem for the spatially homogeneous \\
Landau equation with Shubin class initial datum\\
and Gelfand-Shilov smoothing  effect}

\author[Hao-Guang Li \& Chao-Jiang Xu]
{Hao-Guang Li and Chao-Jiang Xu}

\address{Hao-Guang Li,
\newline\indent
School of Mathematics and Statistics, South-Central University for Nationalities
\newline\indent
430074, Wuhan, P. R. China}
\email{lihaoguang@mail.scuec.edu.cn}

\address{Chao-Jiang Xu,
\newline\indent
School of Mathematics and Statistics, Wuhan University 430072,
Wuhan, P. R. China
\newline\indent
Universit\'e de Rouen, CNRS UMR 6085, Laboratoire de Math\'ematiques Rapha\"el Salem 
\newline\indent
76801 Saint-Etienne du Rouvray, France
}
\email{chao-jiang.xu@univ-rouen.fr}

\date{\today}



\begin{abstract}
In this work, we study the nonlinear spatially homogeneous Landau equation with Maxwellian molecules, by using the spectral analysis, we show that the non linear Landau operators is almost linear, and we prove the existence of weak solution for the Cauchy problem with the initial datum belonging to Shubin space of negative index which conatins the probability measures.  Based on this spectral decomposition, we prove also that the Cauchy problem enjoys $S^{\frac 12}_{\frac 12}$-Gelfand-Shilov smoothing effect, meaning that the weak solution of the Cauchy problem with Shubin class initial datum is ultra-analytics and exponential decay for any positive time.
\end{abstract}

\maketitle


\section{Introduction}\label{S1}
In this work, we study the spatially homogeneous Landau equation
\begin{equation}\label{eq1.10}
\left\{
\begin{array}{ll}
   \partial_t f= Q_L(f,f),\\
  f|_{t=0}=f_0,
\end{array}
\right.
\end{equation}
where $f=f(t,v)\ge0$ is the density distribution function depending on the variables
$v\in\mathbb{R}^{3}$ and the time $t\geq0$. The Landau bilinear collision operator is given by
\begin{equation*}
Q_L(g,f)(v)
=\triangledown_v\cdot\left(\int_{\mathbb{R}^{3}}
a(v-v_*)
\big(g(v_*)(\triangledown_vf)(v)-(\triangledown_vg)(v_*)f(v)\big)
dv_*\right),
\end{equation*}
where $a(v)=(a_{i,j}(v))_{1\leq\,i,j\leq3}$ stands for the non-negative symmetric matrix
$$
a(v)=(|v|^2\textbf{I}-v\otimes\,v)|v|^{\gamma}\in\,M_3(\mathbb{R}),\quad\,-3<\gamma<+\infty.
$$
In this work, we only consider the Cauchy problem \eqref{eq1.10} with the Maxwellian molecules, that means
$\gamma=0$. For the non-negative initial datum $f_0$, we suppose
\begin{equation}\label{f-0}
\int_{\mathbb{R}^3}f_0(v)dv=1, \,\int_{\mathbb{R}^3}v_jf_0(v)dv=0,\,j=1,2,3,\int_{\mathbb{R}^3}|v|^2f_0(v)dv=3.
\end{equation}
We shall study the linearization of the Landau equation \eqref{eq1.10} near the absolute Maxwellian distribution
$$
\mu(v)=(2\pi)^{-\frac 32}e^{-\frac{|v|^{2}}{2}}.
$$
Considering the fluctuation of density distribution function
$$
f(t,v)=\mu(v)+\sqrt{\mu}(v)g(t,v),
$$
since $Q_L(\mu,\mu)=0$, the Cauchy problem \eqref{eq1.10} is reduced to the Cauchy problem
\begin{equation} \label{eq-1}
\left\{ \begin{aligned}
         &\partial_t g+\mathcal{L}(g)={\bf L}(g, g),\,\,\, t>0,\, v\in\mathbb{R}^3,\\
         &g|_{t=0}=g_{0},
\end{aligned} \right.
\end{equation}
with $g_0(v)=\mu^{-\frac 12}f_0(v) -\sqrt{\mu}$, where
$$
\mathcal{L}(g)=-\mu^{-\frac 12}\Big(Q_L(\sqrt{\mu}g,\mu)+Q_L(\mu,\sqrt{\mu}g)\Big),\quad {\bf L}(g, g)=\mu^{-\frac 12}Q_L(\sqrt{\mu}g,\sqrt{\mu}g).
$$
The linear operator $\mathcal{L}$ is non-negative (see \cite{NYKC2}) with the null space
$$
\mathcal{N}=\text{span}\left\{\sqrt{\mu},\,v_1\sqrt{\mu},\,v_2\sqrt{\mu},\,v_3
\sqrt{\mu},\,|v|^2\sqrt{\mu}\right\}.
$$
Then the assumption \eqref{f-0} on the initial datum $f_0$ reduces to
$$
\left\{ \begin{aligned}
         &\int_{\mathbb{R}^3}\sqrt{\mu}(v)g_0(v)dv=0, \\
&\int_{\mathbb{R}^3}v_j\sqrt{\mu}(v)g_0(v)dv=0,\,j=1,2,3,\\
         &\int_{\mathbb{R}^3}|v|^2\sqrt{\mu}(v)g_0(v)dv=0.
\end{aligned} \right.
$$
This shows that $g_0\in \mathcal{N}^{\perp}$.  We recall the spectral decomposition of the linear
Landau operator (see Apendix \ref{appendix} and \cite{Boby}, \cite{NYKC2}).
\begin{equation}\label{1.5+}
\mathcal{L}(\varphi_{n, l, m})=\lambda_{n, l}\, \varphi_{n, l, m},\quad
n,\, l\in\mathbb{N},\,\, -l\leq\,m\leq\,l
\end{equation}
where $\left\{\varphi_{n,l,m}\right\}_{n,l\in\mathbb{N},|m|\leq\,l}$ is an orthonormal basis of $L^2(\mathbb{R}^3)$
composed by eigenvectors of the harmonic oscillator $\mathcal{H}=-\triangle_v +\frac{|v|^2}{4}$  and the Laplace-Beltrami operator on the unit sphere $\mathbb{S}^2$,
\begin{equation*}
\mathcal{H}(\varphi_{n, l, m})=(2n+l+\frac 32)\, \varphi_{n, l, m},\quad -\Delta_{\mathbb{S}^2} (\varphi_{n,l,m}) =l(l+1)\varphi_{n,l,m}.
\end{equation*}
The eigenvalues of \eqref{1.5+} satisfies : $\lambda_{0,0}=\lambda_{0,1}=\lambda_{1,0}=0$,  $\lambda_{0,2}=12$  and for $2n+l>2$,
\begin{equation}\label{lambda}
\lambda_{n, l}=2(2n+l)+l(l+1).
\end{equation}
Using this spectral decomposition, the definition of the operators $e^{c\mathcal{H}^s}$ and $\mathcal{H}^\alpha$ are then classical.

We introduce the following function spaces: Gelfand-Shilov spaces, for $0< s\le 1$,
$$
S^{\frac{1}{2s}}_{\frac{1}{2s}}(\mathbb{R}^3)=\Big\{u\in \mathcal{S}'(\mathbb{R}^3);\,\,  \exists \, c>0,\,
e^{c \mathcal{H}^{s}}u \in L^2(\mathbb{R}^3)\Big\}\, ;
$$
and the Shubin spaces, for $\beta\in \mathbb{R}$, (see \cite{Shubin}, Ch. IV, 25.3),
\[
Q^{\beta}(\mathbb{R}^3)=\Big\{u\in \mathcal{S}'(\mathbb{R}^3);\,\,   \|u\|_{Q^{\beta}(\mathbb{R}^3)} =\bigl\|\mathcal{H}^{\frac{\beta}{2}} \, u\bigr\|_{L^2(\mathbb{R}^3)}
<+\infty\Big\}.
\]
We have
\begin{align*}
Q^{\beta}(\mathbb{R}^3)\subset H^\beta(\mathbb{R}^3),\quad\forall \beta\ge0,\\
H^\beta(\mathbb{R}^3)\subset Q^{\beta}(\mathbb{R}^3) ,\quad\forall \beta <0,
\end{align*}
where $H^\beta(\mathbb{R}^3)$ is the usual Sobolev spaces. In particular, for $\beta<-\frac 32$ the Shubin space 
$Q^\beta(\mathbb{R}^3)$ contains the probability mesures (see \cite{CMWY,MO,MY,TV} and \cite{LX2017}).
See Appendix \ref{appendix} for more properties of Gelfand-Shilov spaces and the Shubin spaces.

It is showed in \cite{MPX} that, for $g_0\in L^2(\mathbb{R}^3)$ with $f_0=\mu+\sqrt{\mu} g_0\ge 0$, the solution of
the Cauchy problem \eqref{eq-1} obtained in \cite{Villani1998-1} belongs to $ S^{\frac{1}{2}}_{\frac{1}{2}}(\mathbb{R}^3)$ for any $t>0$.
In this work, we consider the initial datum which belongs to the Shubin spaces of negative index. The main theorem of this paper is in the following.

\begin{theorem}\label{trick}
Let $\alpha\leq0$, there exists $c_0>0$ such that for any initial datum $g_0\in\, Q^{\alpha}(\mathbb{R}^3)\cap\mathcal{N}^{\perp}$ with
\begin{equation}\label{initial}
\|\mathbb{S}_2g_0\|_{L^2(\mathbb{R}^3)}\leq c_0,
\end{equation}
the Cauchy problem \eqref{eq-1} admits a global weak solution
$$
g\in L^{+\infty}([0, +\infty[; \, Q^{\alpha}(\mathbb{R}^3)).
$$
Moreover, we have the Gelfand-Shilov smoothing effect of Cauchy problem, and there exists $c_1>0$ such that for any $t>0$,
\begin{equation*}
\|e^{c_1 t\mathcal{H}}\mathcal{H}^{\frac{\alpha}{2}}g(t)\|_{L^2(\mathbb{R}^3)}\leq
\|g_0\|_{Q^{\alpha}(\mathbb{R}^3)}.
\end{equation*}
\end{theorem}

\begin{remark}.

1) The orthogonal projectors  $\{\mathbb{S}_N, N\in\mathbb{N}\}$ is defined, for $g\in \mathcal{S}'(\mathbb{R}^3)$,
\begin{equation}\label{SN}
\mathbb{S}_Ng=\sum^{N}_{k=0}\sum_{\substack{2n+l=k}}\sum_{|m|\leq\,l}\langle g,\varphi_{n,l,m}\rangle\varphi_{n,l,m}\in \mathcal{S}(\mathbb{R}^3) .
\end{equation}

2) The constant $c_0$ in \eqref{initial} is not asked to be very small, see the Remark \ref{c-0-1}. On the other hand,
the condition \eqref{initial} is a restriction for the initial datum on $\mathbb{S}_2g_0$, but not a smallness hypothesis for
the initial datum $g_0$.

3) For the Landau equation (also Boltzmann equation), a physics condition on the initial datum is $f_0=\mu+\sqrt{\mu} g_0\ge 0$ which implies the non-negativity of solution $f=\mu+\sqrt{\mu} g$. On the other hand, from the partial differential equations
point of view, for the Cauchy problem \eqref{eq1.10} (also \eqref{eq-1}), we don't need to impose this non-negative condition.
So that, in the Theorem \ref{trick}, we do not ask for the  initial datum  $f_0=\mu+\sqrt{\mu} g_0$ to be non-negative.

4) Combining this Theorem with the results of \cite{Villani1998-1} and \cite{MPX} (see also \cite{MX,Villani1998-2}),  we get a complete result for the Cauchy problem
\eqref{eq-1} with initial datum $g_0\in\, Q^\beta(\mathbb{R}^3)\cap\mathcal{N}^{\perp}, \beta\in\mathbb{R}$:
The existence of global (weak) solution and $S^{\frac{1}{2}}_{\frac{1}{2}}$-Gelfand-Shilov smoothing effect of Cauchy problem.

5) It is well known that  the single Dirac mass on the origin is a stationary solution of the Cauchy problem \eqref{eq1.10}.  The following example is somehow surprise.

\begin{Example}\label{Example}
Let
$$f_0=\delta_0-\left(\frac{3}{2}-\frac{|v|^2}{2}\right)\mu$$
be the initial datum of the Cauchy problem \eqref{eq1.10}, then
$f_0=\mu+\sqrt{\mu}g_0$ with
\begin{equation}\label{ex1}
g_0=\frac{1}{\sqrt{\mu}}\delta_0-\left(\frac{5}{2}-\frac{|v|^2}{2}\right)\sqrt{\mu}\in Q^{\alpha}(\mathbb{R}^3)\cap\mathcal{N}^{\perp},\,\,\,\,\,\alpha<-\frac{3}{2},
\end{equation}
and $\|\mathbb{S}_2g_0\|_{L^2(\mathbb{R}^3)}=0$.  Then Theorem \ref{trick} imply that the Cauchy problem \eqref{eq1.10} admits a global solution
$$
f=\mu+\sqrt{\mu}g\in  L^{+\infty}([0, +\infty[; \, Q^{\alpha}(\mathbb{R}^3) ) \cap
C^{0}(]0, +\infty[; \, S^{\frac{1}{2}}_{\frac{1}{2}}(\mathbb{R}^3) ).
$$
\end{Example}
\end{remark}

This paper is arranged as follows : In the Section \ref{S2}, we introduce the spectral analysis
of the Landau operators and prove that the nonlinear Landau operator is almost diagonal.  By using this decomposition, we can present explicitly the formal solutions to the Cauchy problem  \eqref{eq-1}  by transforming it into an infinite system of ordinary differential equations. In the Section \ref{S3}, we establish an upper bounded estimates for the nonlinear operators. We prove the main theorem \ref{trick} in the Section \ref{S4}, and collect the main technical computations in the Section \ref{S5}. In the Section \ref{appendix}, we give the proof of the Example \ref{Example} and the characterization of the Gelfand-Shilov spaces and the Shubin spaces.

\section{Spectral analysis and formal solutions}\label{S2}

In this section, we study the algebra property of the nonlinear Landau operators on the orthonormal basis
$\{ \varphi_{n,l,m}\}$ of $L^2(\mathbb{R}^3)$,
$$
{\bf L}(\varphi_{\tilde{n},\tilde{l},\tilde{m}}, \varphi_{n,l,m}).
$$
Recall, for $n, l\in\mathbb{N}, m\in\mathbb{Z}, |m|\le l$,
\begin{equation*}
\varphi_{n,l,m}(v)=\left(\frac{n!}{\sqrt{2}\, \Gamma(n+l+3/2)}\right)^{1/2}
\left(\frac{|v|}{\sqrt{2}}\right)^{l}e^{-\frac{|v|^{2}}{4}}
L^{(l+1/2)}_{n}\left(\frac{|v|^{2}}{2}\right)Y^{m}_{l}\left(\frac{v}{|v|}\right),
\end{equation*}
where $\Gamma(\,\cdot\,)$ is the standard Gamma function, and

\noindent
-- $L^{(\alpha)}_{n}$ is the Laguerre polynomial of order $\alpha$ and degree $n$,
\begin{align*}
&L^{(\alpha)}_{n}(x)=\sum^{n}_{r=0}(-1)^{n-r}\frac{\Gamma(\alpha+n+1)}{r!(n-r)!
\Gamma(\alpha+n-r+1)}x^{n-r};
\end{align*}

\noindent
-- $Y^{m}_{l}(\sigma)$ is the orthonormal basis of spherical harmonics
\begin{equation*}
Y^{m}_{l}(\sigma)=N_{l,m}P^{|m|}_{l}(\cos\theta)e^{im\phi},\,\,|m|\leq l,
\end{equation*}
where $\sigma=(\cos\theta,\sin\theta\cos\phi,\sin\theta\sin\phi)$ and $N_{l,m}$ is  the normalisation factor.
It is obviously that, the conjugate of $Y^{m}_{l}(\sigma)$ satisfies
$$
\overline{Y^{m}_{l}(\sigma)}=Y^{-m}_{l}(\sigma).
$$

\noindent
-- $P^{|m|}_{l}$~is the Legendre functions of the first kind of order $l$ and degree $|m|$
$$
P^{|m|}_{l}(x)= (1-x^2)^\frac{|m|}{2}\frac{d^{|m|}}{dx}\left(\frac{1}{2^ll!}\frac{d^l}{dx^l}(x^2-1)^l\right).
$$
Then, $\{ \varphi_{n,l,m}\}\subset\mathcal{S}(\mathbb{R}^3)$ the Schwartz function space, and
\begin{align*}
&\varphi_{0,0,0}(v)=\sqrt{\mu},\,\qquad\qquad\,\, \varphi_{0,1,0}(v)=v_1\sqrt{\mu},\\
&\varphi_{0,1,1}(v)=\frac{v_2+iv_3}{\sqrt{2}}\sqrt{\mu},\,\,\,\,\,
\varphi_{0,1,-1}(v)=\frac{v_2-iv_3}{\sqrt{2}}\sqrt{\mu},\\
&\varphi_{1,0,0}(v)=\sqrt{\frac{2}{3}}\left(\frac{3}{2}-\frac{|v|^2}{2}\right)\sqrt{\mu}\, ,
\end{align*}
and
$$
\mathcal{N}=\text{span}\left\{\varphi_{0,0,0},\,\varphi_{0,1,0},\,\varphi_{0,1,1},\,
\varphi_{0,1,-1},\,\varphi_{1,0,0}\right\}.
$$
We have also the explicit form of the eigenfunctions $\{\varphi_{0,2,m_2}, |m_2|\le2\}$ :
\begin{equation}\label{S22}
\begin{array}{l}
\varphi_{0,2,0}(v)=\sqrt{\frac{1}{3}}\left(\frac{3}{2}v_1^2-\frac{1}{2}|v|^2\right)\sqrt{\mu},\quad
\varphi_{0,2,1}(v)
=\frac{v_1 v_2+iv_1v_3}{\sqrt{2}}\sqrt{\mu},\\
\varphi_{0,2,-1}(v)=\frac{v_1 v_2-iv_1v_3}{\sqrt{2}}\sqrt{\mu},\quad
\varphi_{0,2,2}(v)
=\left(\frac{v_2^2-v^2_3}{2\sqrt{2}}+i\frac{v_2v_3}{\sqrt{2}}\right)\sqrt{\mu},\\
\varphi_{0,2,-2}(v)
=\left(\frac{v_2^2-v^2_3}{2\sqrt{2}}-i\frac{v_2v_3}{\sqrt{2}}\right)\sqrt{\mu}.
\end{array}
\end{equation}

We have the following algebraic identities :

\begin{proposition}\label{expansion}
For $n,l\in\mathbb{N}$, $|m|\leq\,l$, we have
\begin{align*}
&(i)\quad{\bf L}(\varphi_{0,0,0},\varphi_{n,l,m})=-\left(2(2n+l)+l(l+1)\right)\varphi_{n,l,m};\\
&(ii)\quad{\bf L}(\varphi_{0,1,m_1},\varphi_{n,l,m})\\
&\qquad=A^-_{n,l,m,m_1}\varphi_{n+1,l-1,m_1+m}+A^+_{n,l,m,m_1}\varphi_{n,l+1,m_1+m}, \,\forall\,|m_1|\leq1;\\
&(iii)\quad{\bf L}(\varphi_{1,0,0},\varphi_{n,l,m})=\frac{4\sqrt{3(n+1)(2n+2l+3)}}{3}
\varphi_{n+1,l,m};\\
&(iv)\quad{\bf L}(\varphi_{0,2,m_2},\varphi_{n,l,m})
=A^1_{n,l,m,m_2}
\varphi_{n+2,l-2,m+m_2}\\
&\qquad\quad+A^2_{n,l,m,m_2}\varphi_{n+1,l,m+m_2}+A^3_{n,l,m,m_2}\varphi_{n,l+2,m+m_2}, \,\forall\,|m_2|\leq2;\\
&(v)\quad{\bf L}(\varphi_{\tilde n,\tilde{l},\tilde{m}},\varphi_{n,l,m})=0, \,\forall\,2\tilde n+\tilde{l}>2,\, |\tilde m|\le \tilde l.
\end{align*}
where the coefficients will be precisely defined in Section \ref{S5}.
\end{proposition}

The proof of this Proposition and the estimates of $A^1_{n,l,m,m_2}, A^2_{n,l,m,m_2}$ and $A^3_{n,l,m,m_2}$ are the main technic parts of this paper, we will give it in the Section \ref{S5}.

Now we come back to the Cauchy problem \eqref{eq-1}, we search a solution of the form
\begin{equation}\label{g-S}
g(t)=\sum^{+\infty}_{n=0}\sum^{+\infty}_{l=0}\sum^{l}_{m=-l}g_{n,l,m}(t)\varphi_{n,l,m},\, \,\,\, g_{n,l,m}(t)=\left\langle g(t),\, \varphi_{n,l,m}\right\rangle
\end{equation}
with initial data
$$
g|_{t=0}=g_0=\sum^{+\infty}_{n=0}\sum^{+\infty}_{l=0}\sum^{l}_{m=-l}g^0_{n,l,m}\varphi_{n,l,m},\,\,\,\,\,g^0_{n,l,m}=\left\langle g_0,\, \varphi_{n,l,m}\right\rangle.
$$
The hypothesis $g_0\in Q^{\alpha}(\mathbb{R}^3)\cap\mathcal{N}^{\perp}$ is equivalent to,
$$
g^0_{0,0,0}=g^0_{0,1,1}=g^0_{0,1,0}=g^0_{0,1,-1}=g^0_{1,0,0}=0,
$$
and
$$
\|g_0\|^2_{Q^\alpha}=\sum^{+\infty}_{n=0}\sum^{+\infty}_{l=0}\sum^{l}_{m=-l}(2n+l+\frac 32)^\alpha\,
|g^0_{n,l,m}|^2<\infty.
$$
See Appendix \ref{appendix} for the norm of Shubin space.

It follows from Proposition \ref{expansion} that,  we have the almost diagonalization of non linear 
Landau operators, meaning that for the function $f, g$ define by the series \eqref{g-S}, for $n,l\in\mathbb{N}$, $m\in\mathbb{Z}$, $|m|\leq\,l$
\begin{equation}\label{Ln}
\begin{split}
&\Big({\bf L}(f, g), \varphi_{n,l,m}\Big)_{L^2}\\
=&-\left(2(2n+l)+l(l+1)\right)f_{0,0,0}(t)g_{n,l,m}(t)\\
&\quad+
\sum_{\substack{|m^*|\leq\,l+1,|m_1|\leq1\\m_1+m^*=m}}A^-_{n-1,l+1,m^*,m_1}f_{0,1,m_1}(t)g_{n-1,l+1,m^*}(t)\\
&\quad+
\sum_{\substack{|m^*|\leq\,l-1,|m_1|\leq1\\m_1+m^*=m}}A^+_{n,l-1,m^*,m_1}f_{0,1,m_1}(t)g_{n,l-1,m^*}(t)\\
&\quad+\frac{4\sqrt{3n(2n+2l+1)}}{3}f_{1,0,0}(t)g_{n-1,l,m}(t)\\
&\quad+\sum_{\substack{|m^*|\leq\,l+2,|m_2|\leq2\\m^*+m_2=m}}
A^1_{n-2,l+2,m^*,m_2}f_{0,2,m_2}(t)g_{n-2,l+2,m^*}(t)\\
&\quad+\sum_{\substack{|m^*|\leq\,l,|m_2|\leq2\\m^*+m_2=m}}A^2_{n-1,l,m^*,m_2}f_{0,2,m_2}(t)g_{n-1,l,m^*}(t)\\
&\quad+\sum_{\substack{|m^*|\leq\,l-2,|m_2|\leq2\\m^*+m_2=m}}A^3_{n,l-2,m^*,m_2}f_{0,2,m_2}(t)g_{n,l-2,m^*}(t),
\end{split}
\end{equation}
with the conventions
$$
g_{n, l, m}\equiv 0, \,\,\mbox{if}\,\, n<0 \,\,\mbox{or}\,\, l<0\,,
$$
and
\begin{equation*}
\Big( \mathcal{L}( g ), \varphi_{n,l,m}\Big)_{L^2} =\lambda_{n,l}\,g_{n,l,m}(t),\,\,\,\,n,\, l\in\mathbb{N},\, m\in\mathbb{Z},\, |m|\leq\,l.
\end{equation*}
We remark from \eqref{Ln} that,
\begin{equation}\label{L-perp}
\forall\,f,g\in\mathcal{N}^{\perp}\Rightarrow\quad {\bf L}(f, g)\in\mathcal{N}^{\perp}.
\end{equation}
So that, formally, if $g$ is a solution of the Cauchy problem \eqref{eq-1},
we find that the family of functions $\{g_{n,l,m}(t); n, l\in\mathbb{N}, |m|\leq\,l\}$, satisfy the following infinite system
of the differential equations, $n,l\in\mathbb{N}, |m|\leq\,l$,
\begin{equation} \label{ODE-0}
\left\{ \begin{aligned}
         &\partial_t g_{n,l,m}(t)+\lambda_{n,l}\,g_{n,l,m}(t)=\big({\bf L}(g, g), \varphi_{n,l,m}\big)_{L^2},\,\,\,\, t>0;\\
         &g_{n,l,m}|_{t=0}=\left\langle g_0,\varphi_{n,l,m}\right\rangle=g^0_{n,l,m}
\end{aligned} \right.
\end{equation}
where $\big({\bf L}(g, g), \varphi_{n,l,m}\big)_{L^2}$ was precisely defined in \eqref{Ln}. 
We have firstly,

\begin{proposition}\label{SSN2}
Let  $g_0\in Q^{\alpha}(\mathbb{R}^3)\cap \mathcal{N}^{\perp} $, assume that $g$ is a solution of the Cauchy problem \eqref{eq-1}
of the form \eqref{g-S}, then we have
\begin{equation}\label{N}
g_{0,0,0}(t)=g_{0,1,0}(t)=g_{0,1,1}(t)=g_{0,1,-1}(t)=g_{1,0,0}(t)=0,\quad \forall\,t\ge 0,
\end{equation}
and
\begin{equation}\label{N2}
g_{0,2,m}(t)=e^{-12t}g^0_{0,2,m},\quad t\ge 0,\,\,\,|m|\le 2.
\end{equation}
\end{proposition}
\begin{proof}
(1) Substituting $n=0,l=0,m=0$ into the above infinite ODE system \eqref{ODE-0}, one has
$$
\partial_tg_{0,0,0}(t)+\lambda_{0,0}g_{0,0,0}(t)=0.
$$
We remind that $\lambda_{0,0}=0$, then
$$
g_{0,0,0}(t)=g^0_{0,0,0}=0.
$$
(2) Now we set $n=0$, $l=1$, and $|m|\leq1$, the ODE system \eqref{ODE-0} turn out to be
\begin{align*}
&\partial_t g_{0,1,m}(t)+\lambda_{0,1}\,g_{0,1,m}(t)\\
&=-4g_{0,0,0}(t)g_{0,1,m}(t)+A^+_{0,0,0,m}g_{0,1,m}(t)g_{0,0,0}(t)
\end{align*}
By using the known results
$$
\lambda_{0,1}=0,\quad g_{0,0,0}(t)=0,
$$
one can verify that
$$
g_{0,1,m}(t)=g^0_{0,1,m}=0,\quad\forall|m|\leq1.
$$
(3) Take now $n=1, l=0, m=0$ in \eqref{ODE-0}, we have
\begin{align*}
&\partial_t g_{1,0,0}(t)+\lambda_{1,0}\,g_{1,0,0}(t)\\
&=-4g_{0,0,0}(t)g_{1,0,0}(t)+
\sum_{\substack{|m^*|\leq\,1,|m_1|\leq1\\m_1+m^*=0}}A^-_{0,1,m^*,m_1}g_{0,1,m_1}(t)g_{0,1,m^*}(t)\\
&\quad+4g_{1,0,0}(t)g_{0,0,0}(t)+A^2_{0,0,0,0}g_{0,2,0}(t)g_{0,0,0}(t).
\end{align*}
Then
$$
\lambda_{1,0}=0,\quad\,g_{0,0,0}(t)=0,\quad\,g_{0,1,m'}(t)=0,\,\,\,\,\forall |m'|\leq1,
$$
imply
$$
g_{1,0,0}(t)=g^0_{1,0,0}=0.
$$
(4) Furthermore, for $n=0,l=2$ and $|m|\leq2$ in \eqref{ODE-0}, we have that
$$
g_{0,0,0}(t)=0,\,\,\,\, g_{0,1,m'}(t)=0,\quad\,\forall |m'|\leq1,
$$
imply
\begin{align*}
&\partial_t g_{0,2,m}(t)+\lambda_{0,2}\,g_{0,2,m}(t)=0.
\end{align*}
Recalled that $\lambda_{0,2}=12$ in \eqref{lambda}, we obtain,
$$g_{0,2,m}(t)=e^{-12t}g^0_{0,2,m},\quad\forall |m|\leq2.$$
This ends the proof of Proposition \ref{SSN2}.
\end{proof}

Substituting \eqref{N} and \eqref{N2} into the infinite system
of the differential equations \eqref{ODE-0},  we have, for all $2n+l > 2, |m|\leq\,l$,
\begin{equation}\label{ODE-1}
\left\{ \begin{aligned}
         &\partial_t g_{n,l,m}(t)+\lambda_{n,l}\,g_{n,l,m}(t)=\\
&\quad\quad+\sum_{\substack{|m^*|\leq\,l+2,|m_2|\leq2\\m^*+m_2=m}}
A^1_{n-2,l+2,m^*,m_2}e^{-12t}g^0_{0,2,m_2}\, g_{n-2,l+2,m^*}(t)\\
&\quad\quad+\sum_{\substack{|m^*|\leq\,l,|m_2|\leq2\\m^*+m_2=m}}A^2_{n-1,l,m^*,m_2}e^{-12t}g^0_{0,2,m_2}\, g_{n-1,l,m^*}(t)\\
&\quad\quad+\sum_{\substack{|m^*|\leq\,l-2,|m_2|\leq2\\m^*+m_2=m}}A^3_{n,l-2,m^*,m_2}e^{-12t}g^0_{0,2,m_2}\, g_{n,l-2,m^*}(t),\\
         &g_{n,l,m}|_{t=0}=g^0_{n,l,m},
\end{aligned} \right.
\end{equation}
with the convention
\begin{equation}\label{conv}
A^1_{n-2,l+2,m^*,m_2}=0, \,\,\mbox{if}\,\, n-2<0;\quad  A^2_{n-1,l, m^*,m_2}=0,\,\,\mbox{if}\,\,n-1<0.
\end{equation}
We can solve this infinite differential equation by induction.

In fact, for $n=0, l \ge 3, |m|\leq\,l$, the following system
$$
\left\{ \begin{aligned}
         &\partial_t g_{0,l,m}(t)+\lambda_{0,l}\,g_{0,l,m}(t)\\
&=\sum_{\substack{|m^*|\leq\,l-2,|m_2|\leq2\\m^*+m_2=m}}A^3_{0,l-2,m^*,m_2}e^{-12t}g^0_{0,2,m_2}\,g_{0,l-2,m^*}(t),\\
         &g_{0,l,m}(0)= g^0_{0,l,m}.
\end{aligned} \right.
$$
can be solved by induction on $l$ start from $l=3$ since $g_{0,1,m^*}(t)\equiv 0$ for all  $|m^*|\le 1$. For the general case of $l$,
the index of the right hand side are $l-2$, which have been already known by induction.

Then we solve the differential equations \eqref{ODE-1} for all $n\ge 1, l\ge 0$ and $|m|\leq\,l$. We also prove
by induction on $n$ and for fixed $n$ induction on $l$. Since for the first two terms on the right hand side of  \eqref{ODE-1}, the first index are less than $n-1$, and  for the last terms on the right hand side, the second index are less than $l-2$, which have been already known by induction. So that in each steps
of the induction, the right hand side of  \eqref{ODE-1}  is  already known by induction hypothesis. Then the differential equations are linear differential equations, and can be solved explicitly with any initial detum $g^0_{n,l,m}$. We get then the formal solution of Cauchy problem \eqref{eq-1} by solve the differential system \eqref{ODE-1}, and we have :

\begin{theorem}\label{th:ODE}
Let $\{g^0_{n,l,m};\,\, n, l\in\mathbb{N}, |m|\le l\}$ be a complex sequence with
$$
g^0_{0,0,0}=g^0_{0,1,1}=g^0_{0,1,0}=g^0_{0,1,-1}=g^0_{1,0,0}=0.
$$
Then the system \eqref{ODE-1} admits a sequence of solutions $\{g_{n,l,m}(t);\, 2n+l>2, |m|\le l\}$.

For all $N\ge2$, we note that
\begin{equation}\label{gN}
g_N(t)=\sum^{N}_{k=2}\sum_{\substack{2n+l=k\\n+l\ge2}} \sum_{|m|\le\,l}g_{n,l,m}(t)\varphi_{n,l,m}
\end{equation}
with
$$
g_{0,2,m}(t)=e^{-12t}g^0_{0,2,m},\,\,\,\,|m|\le 2,\,\,\, t>0.
$$
Then $g_N$ satisfies the following Cauchy problem
\begin{equation}\label{eq-2}
\left\{ \begin{aligned}
         &\partial_t g_N+\mathcal{L}(g_N)=\mathbb{S}_{N}{\bf L}(g_N, g_N),\,\\
         &g_{N}|_{t=0}=\sum_{\substack{2\leq2n+l\leq\,N\\n+l\ge2}} \sum_{|m|\le\,l} g^0_{n,l,m}\varphi_{n,l,m}.
\end{aligned} \right.
\end{equation}
\end{theorem}

The proof of the existence of weak solution of Theorem \ref{trick} is reduced to prove the convergence of the sequences $\{g_N; \,N\in \mathbb{N}\}$ in the function space $Q^{\alpha}(\mathbb{R}^3)$. Namely,
$$
g_N\rightarrow\,g(t)=\sum^{+\infty}_{k=2}\sum_{\substack{2n+l=k\\n+l\ge2}} \sum_{|m|\le\,l}g_{n,l,m}(t)\varphi_{n,l,m}\in\,Q^{\alpha}(\mathbb{R}^3),\quad\text{as}\,\, N\rightarrow+\infty.
$$
The Gelfand-Shilov regularity is reduced to prove: there exists a constant $c_1>0$, such that
\begin{equation*}
\forall t>0,\,\,\,\,
\|e^{c_1t\mathcal{H}}\mathcal{H}^{\frac{\alpha}{2}}g(t)\|^2_{L^2(\mathbb{R}^3)}
=\sum e^{c_1t (2n+l+\frac{3}{2})}(2n+l+\frac{3}{2})^{\alpha}|g_{n,l,m}(t)|^2<\infty.
\end{equation*}
This will be the main jobs of the Section \ref{S3} and Section \ref{S4}.


\section{The trilinear estimates for non linear operator}\label{S3}

To prove the convergence of the formal solution obtained in Theorem \ref{th:ODE}, we need to estimate the following trilinear terms
$$
\left({\bf L}(f,g),h\right)_{L^2(\mathbb{R}^3)},
\,\,\,f,g,h\in\mathscr{S}(\mathbb{R}^3)\cap \mathcal{N}^{\perp}\, .
$$
We need firstly the following estimates for the coefficients $A^1, A^2$ and $A^3$ (see their definition \eqref{An} in Section \ref{S5}) of the Proposition \ref{expansion}.

\begin{proposition}\label{recur}
For the coefficients of the Proposition \ref{expansion} defined in \eqref{An}, we have the following estimates:

\noindent
1) For $n,l\in\mathbb{N}$, $n\ge2$,
\begin{equation}\label{A1}
\max_{|m^*|\leq l}\sum_{\substack{|m|\leq\,l+2, |m_2|\leq2\\m+m_2=m^*}}
\left|A^1_{n-2,l+2,m,m_2}\right|^2\leq\frac{16n(n-1)}{3}.
\end{equation}

\noindent
2) For $n,l\in\mathbb{N}$, $n\ge1$,
\begin{align}\label{A2}
 A^2_{n-1,0,0,0}&=0;\nonumber\\
\max_{|m^*|\leq l}\sum_{\substack{|m|\leq\,l, |m_2|\leq2\\m+m_2=m^*}}
\Big|A^2_{n-1,l,m,m_2}\Big|^2&\leq\frac{4n(2n+2l+1)}{3},\,\forall\,\,l\geq1.
\end{align}
\noindent
3) For $n,l\in\mathbb{N}$, $l\ge2$,
\begin{equation}\label{A3}
\max_{|m^*|\leq l}\sum_{\substack{|m|\leq\,l-2, |m_2|\leq2\\m+m_2=m^*}}
\left|A^3_{n,l-2,m,m_2}\right|^2\leq\frac{(2n+2l+1)(2n+2l-1)}{2}.
\end{equation}
\end{proposition}

We will give the proof of this Proposition in the Section \ref{S5}.

We now present the trilinear estimation for the nonlinear Landau operator ${\bf L}$,
for $g\in\mathcal{S}'(\mathbb{R}^3)\cap \mathcal{N}^{\perp}, N>2$, we note
\begin{equation}\label{SN2}
\tilde{\mathbb{S}}_N g
=\sum_{\substack{2\leq2n+l\leq\,N\\n+l\geq2}}\sum_{|m|\leq\,l}g_{n,l,m}\, \varphi_{n,l,m},\quad
g_{n,l,m}=\langle g,\, \varphi_{n,l,m}\rangle.
\end{equation}
Then we have the following trilinear estimates:

\begin{proposition}\label{estimatetri}
Let $f, g,h\in Q^{\alpha}(\mathbb{R}^3)\cap \mathcal{N}^{\perp}$ with $\alpha\leq0$, then for any $N\ge 2$,
\begin{align*}
&|(\mathbf{L}(\tilde{\mathbb{S}}_N f, \, \tilde{\mathbb{S}}_N g),\mathcal{H}^{\alpha}\tilde{\mathbb{S}}_N h)_{L^2}|
\\
\leq&  \left(\frac{4\sqrt{3}}{3}+\sqrt{2}\right)\|\tilde{\mathbb{S}}_2f\|_{L^2}
\|\mathcal{H}^{\frac{\alpha+1}{2}}\tilde{\mathbb{S}}_{N-2} g\|_{L^2}
\|\mathcal{H}^{\frac{\alpha+1}{2}}\tilde{\mathbb{S}}_N h\|_{L^2},
\end{align*}
and also for any $c>0$, $t\geq0$,
\begin{align*}
&|({\bf L}(\tilde{\mathbb{S}}_Nf, \, \tilde{\mathbb{S}}_N g),e^{2ct\mathcal{H}}\mathcal{H}^{\alpha}\tilde{\mathbb{S}}_N h)_{L^2}|\notag\\
\leq& \left(\frac{4\sqrt{3}}{3}+\sqrt{2}\right)e^{2ct}\|\tilde{\mathbb{S}}_2f\|_{L^2}
\|e^{ct \mathcal{H}}\mathcal{H}^{\frac{\alpha+1}{2}}\tilde{\mathbb{S}}_{N-2}g\|_{L^2}
\|e^{ct \mathcal{H}}\mathcal{H}^{\frac{\alpha+1}{2}}\tilde{\mathbb{S}}_N h\|_{L^2}\,.
\end{align*}
\end{proposition}

The proof of this Proposition is similar to Lemma 3.5 in \cite{NYKC3}, Proposition 3.2 in \cite{LX2017} and Section 3  in \cite{GLX}.

\begin{proof}
Let $f, g,h\in Q^{\alpha}(\mathbb{R}^3) \cap \mathcal{N}^{\perp}$ with $\alpha\leq0$.   For $N\geq2$,
by using the orthogonal property of $\{\varphi_{n,l,m}; \, n,l\in\mathbb{N},  |m|\leq\,l\}$,
we can deduce from Proposition \ref{expansion}  and \eqref{conv} that
\begin{align*}
&({\bf L}(\tilde{\mathbb{S}}_Nf, \, \tilde{\mathbb{S}}_N g),\mathcal{H}^{\alpha}\tilde{\mathbb{S}}_N h)\\
=&\sum_{\substack{2\leq2n+l\leq\,N\\n\geq2}}\sum_{\substack{|m|\leq\,l+2,|m_2|\leq2\\|m+m_2|\leq l}}
A^1_{n-2,l+2,m,m_2}(2n+l+\frac{3}{2})^{\alpha}f_{0,2,m_2}g_{n-2,l+2,m}h_{n,l,m+m_2}\\
&+
\sum_{\substack{2\leq2n+l\leq\,N\\n\geq1,n+l\geq2}}\sum_{\substack{|m|\leq\,l,|m_2|\leq2\\|m+m_2|\leq\,l}}
 A^2_{n-1,l,m,m_2}(2n+l+\frac{3}{2})^{\alpha}f_{0,2,m_2}g_{n-1,l,m}h_{n,l,m+m_2}
\\
&+\sum_{\substack{2\leq2n+l\leq\,N\\l\geq2}}\sum_{\substack{|m|\leq\,l-2,|m_2|\leq2\\|m+m_2|\leq\,l}} A^3_{n,l-2,m,m_2}(2n+l+\frac{3}{2})^{\alpha}f_{0,2,m_2}g_{n,l-2,m}h_{n,l,m+m_2}\\
&\le \mathbf{B}_1+\mathbf{B}_2+\mathbf{B}_3.
\end{align*}
For the first term $\mathbf{B}_1$, we have
\begin{align*}
\mathbf{B}_1
\leq&\sum_{\substack{2\leq2n+l\leq\,N\\n\geq2}}(2n+l+\frac{3}{2})^{\alpha}\\
&\qquad\times\sum_{|m_2|\leq2}\sum_{\substack{|m|\leq\,l+2\\|m+m_2|\leq\,l}}
|f_{0,2,m_2}|\Big|
A^1_{n-2,l+2,m,m_2}g_{n-2,l+2,m}h_{n,l,m+m_2}\Big|\,,
\end{align*}
by using the Cauchy-Schwarz inequality
\begin{align*}
&\sum_{|m_2|\leq2}\Big(|f_{0,2,m_2}|\sum_{\substack{|m|\leq\,l+2\\|m+m_2|\leq\,l}}\Big|
A^1_{n-2,l+2,m,m_2}g_{n-2,l+2,m}h_{n,l,m+m_2}\Big|\Big)\\
\leq&\|\tilde{\mathbb{S}}_2f\|_{L^2}\Big(\sum_{|m_2|\leq2}\Big(\sum_{\substack{|m|\leq\,l+2\\|m+m_2|\leq\,l}}\big|
A^1_{n-2,l+2,m,m_2}g_{n-2,l+2,m}h_{n,l,m+m_2}\big|\Big)^2\Big)^{\frac{1}{2}}\\
\le &\|\tilde{\mathbb{S}}_2f\|_{L^2}\Big(\sum_{|m|\leq\,l+2}|g_{n-2,l+2,m}|^2\Big)^{\frac{1}{2}}\\
&\quad\times\Big(\sum_{|m_2|\leq2}\sum_{\substack{|m|\leq\,l+2\\|m+m_2|\leq\,l}}\Big|A^1_{n-2,l+2,m,m_2}h_{n,l,m+m_2}\Big|^2\Big)^{\frac{1}{2}}.
\end{align*}
By changing the order of summation
$$
\sum_{\substack{|m_2|\leq2,|m|\leq\,l+2\\|m+m_2|\leq\,l}}=\sum_{|m^*|\leq l}\,\,\sum_{\substack{|m_2|\leq2,|m|\leq\,l+2\\m+m_2=m^*}}
$$
and using \eqref{A1} in Proposition \ref{recur}, we have
\begin{align*}
&\sum_{|m_2|\leq2}\sum_{\substack{|m|\leq\,l+2\\|m+m_2|\leq\,l}}\Big|A^1_{n-2,l+2,m,m_2}h_{n,l,m+m_2}\Big|^2\\
&=\sum_{|m^*|\leq l}\Big|h_{n,l,m^*}\Big|^2\Big(\sum_{\substack{|m|\leq\,l+2, |m_2|\leq2\\m+m_2=m^*}}\Big|A^1_{n-2,l+2,m,m_2}|^2\Big)\\
&\leq\frac{16n(n-1)}{3}\sum_{|m^*|\leq l}\Big|h_{n,l,m^*}\Big|^2
\end{align*}
Substituting back to the estimation of $\mathbf{B}_1$, one can verify that
\begin{align*}
\mathbf{B}_1&\leq\|\tilde{\mathbb{S}}_2f\|_{L^2}
\sum_{\substack{2\leq2n+l\leq\,N\\n\geq2}}(2n+l+\frac{3}{2})^{\alpha}\sqrt{\frac{16n(n-1)}{3}}\\
&\qquad\times\Big(\sum_{|m|\leq\,l+2}|g_{n-2,l+2,m}|^2\Big)^{\frac{1}{2}}
\Big(\sum_{|m^*|\leq\,l} |h_{n,l,m^*}|^2\Big)^{\frac{1}{2}}\\
&\leq\frac{2\sqrt{3}}{3}\|\tilde{\mathbb{S}}_2f\|_{L^2}\Big(\sum_{\substack{2\leq2n+l\leq\,N\\n\geq2}}(2n+l-\frac{1}{2})^{\alpha+1}
\sum_{|m|\leq\,l+2}|g_{n-2,l+2,m}|^2\Big)^{\frac{1}{2}}\\
&\qquad\times\Big(\sum_{\substack{2\leq2n+l\leq\,N\\n\geq2}}(2n+l+\frac{3}{2})^{\alpha+1}
\sum_{|m^*|\leq\,l} |h_{n,l,m^*}|^2\Big)^{\frac{1}{2}}\\
&\leq\frac{2\sqrt{3}}{3}\|\tilde{\mathbb{S}}_2f\|_{L^2}
\|\mathcal{H}^{\frac{\alpha+1}{2}}\tilde{\mathbb{S}}_{N-2} g\|_{L^2}
\|\mathcal{H}^{\frac{\alpha+1}{2}}\tilde{\mathbb{S}}_N h\|_{L^2},
\end{align*}
where we use the estimation  $(2n+l+\frac{3}{2})^{\alpha}(2n-2)\leq(2n+l-\frac{1}{2})^{\alpha+1}$ when $\alpha\le 0$.

Now we turn back to estimate $\mathbf{B}_2,\mathbf{B}_3$.  By using the Cauchy-Schwarz inequality
\begin{align*}
&\sum_{|m_2|\leq2}\Big(|f_{0,2,m_2}|\sum_{\substack{|m|\leq\,l\\|m+m_2|\leq\,l}}\Big|
A^2_{n-1,l,m,m_2}g_{n-1,l,m}h_{n,l,m+m_2}\Big|\Big)\\
\leq&\|\tilde{\mathbb{S}}_2f\|_{L^2}\Big(\sum_{|m|\leq\,l}|g_{n-1,l,m}|^2\Big)^{\frac{1}{2}}\Big(\sum_{\substack{|m|\leq\,l,|m_2|\leq2\\|m+m_2|\leq\,l}}
\Big|A^2_{n-1,l,m,m_2}\Big|^2\Big|h_{n,l,m+m_2}\Big|^2\Big)^{\frac{1}{2}}\\
\le &\|\tilde{\mathbb{S}}_2f\|_{L^2}\Big(\sum_{|m|\leq\,l}|g_{n-1,l,m}|^2\Big)^{\frac{1}{2}}\\
&\qquad\qquad\times\Big(\sum_{|m^*|\leq\,l}\Big|h_{n,l,m^*}\Big|^2\Big(\sum_{\substack{|m|\leq\,l,|m_2|\leq2\\m+m_2=m^*}}
\Big|A^2_{n-1,l,m,m_2}\Big|^2\Big)\Big)^{\frac{1}{2}},
\end{align*}
and
\begin{align*}
&\sum_{|m_2|\leq2}\Big(|f_{0,2,m_2}|
\sum_{\substack{|m|\leq\,l-2\\|m+m_2|\leq\,l }}\Big|
A^3_{n,l-2,m,m_2}g_{n,l-2,m}h_{n,l,m+m_2}\Big|\Big)\\
\leq&\|\tilde{\mathbb{S}}_2f\|_{L^2}\Big(\sum_{|m|\leq\,l-2}|g_{n,l-2,m}|^2\Big)^{\frac{1}{2}}
\Big(\sum_{\substack{|m|\leq\,l-2,|m_2|\leq2\\|m+m_2|\leq\,l }}
\Big|A^3_{n,l-2,m,m_2}h_{n,l,m+m_2}\Big|^2\Big)^{\frac{1}{2}}\\
\le &\|\tilde{\mathbb{S}}_2f\|_{L^2}\Big(\sum_{|m|\leq\,l-2}|g_{n,l-2,m}|^2\Big)^{\frac{1}{2}}\\
&\qquad\qquad\times\Big(\sum_{|m^*|\leq\,l}\Big|h_{n,l,m^*}\Big|^2
\Big(\sum_{\substack{|m|\leq\,l-2,|m_2|\leq2\\m+m_2=m^* }}
\Big|A^3_{n,l-2,m,m_2}\Big|^2\Big)
\Big)^{\frac{1}{2}}.
\end{align*}
Substituting the estimations \eqref{A2} and \eqref{A3} in $\mathbf{B}_2$, $\mathbf{B}_3$, it follows that
\begin{align*}
\mathbf{B}_2&\leq\|\tilde{\mathbb{S}}_2f\|_{L^2}
\sum_{\substack{2\leq2n+l\leq\,N\\n\geq1,l\geq1}}(2n+l+\frac{3}{2})^{\alpha}\sqrt{\frac{4n(2n+2l+1)}{3}}\\
&\qquad\times\Big(\sum_{|m|\leq\,l}|g_{n-1,l,m}|^2\Big)^{\frac{1}{2}}
\Big(\sum_{|m^*|\leq\,l} |h_{n,l,m^*}|^2\Big)^{\frac{1}{2}}\\
&\leq\frac{2\sqrt{3}}{3}\|\tilde{\mathbb{S}}_2f\|_{L^2}
\|\mathcal{H}^{\frac{\alpha+1}{2}}\tilde{\mathbb{S}}_{N-2} g\|_{L^2}
\|\mathcal{H}^{\frac{\alpha+1}{2}}\tilde{\mathbb{S}}_N h\|_{L^2},
\end{align*}
here for $n\ge1$, we use
\begin{align*}
&(2n+l+\frac32)^\alpha (n+l+\frac{1}{2})\leq (2n+l+\frac32)^{\alpha+1};\\
&2n(2n+l+\frac32)^\alpha \le (2n+l-\frac12)^{\alpha+1},\quad \text{for}\,\,l\ge1,\,\alpha\le0.
\end{align*}
And
\begin{align*}
\mathbf{B}_3&\leq\|\tilde{\mathbb{S}}_2f\|_{L^2}
\sum_{\substack{2\leq2n+l\leq\,N\\l\geq2}}(2n+l+\frac{3}{2})^{\alpha}\sqrt{\frac{(2n+2l+1)(2n+2l-1)}{2}}\\
&\qquad\times\left(\sum_{|m|\leq\,l-2}|g_{n,l-2,m}|^2\right)^{\frac{1}{2}}
\left(\sum_{|m^*|\leq\,l} |h_{n,l,m^*}|^2\right)^{\frac{1}{2}}\\
&\leq\sqrt{2}\|\tilde{\mathbb{S}}_2f\|_{L^2}
\|\mathcal{H}^{\frac{\alpha+1}{2}}\tilde{\mathbb{S}}_{N-2} g\|_{L^2}
\|\mathcal{H}^{\frac{\alpha+1}{2}}\tilde{\mathbb{S}}_N h\|_{L^2}\, ,
\end{align*}
here for $l\ge 2$ and $n\in\mathbb{N}$, we use
\begin{align*}
&(2n+l+\frac32)^\alpha(n+l+\frac{1}{2})\leq(2n+l+\frac32)^{\alpha+1};\\
&(2n+l+\frac32)^\alpha(n+l-\frac{1}{2})\leq(2n+l-\frac12)^{\alpha+1}, \quad \text{for}\,\,\alpha\le0.
\end{align*}
Therefore,
\begin{align*}
&|({\bf L}(\tilde{\mathbb{S}}_N f, \, \tilde{\mathbb{S}}_N g),\mathcal{H}^{\alpha}\tilde{\mathbb{S}}_N h)_{L^2}|\\
&\leq \left(\frac{4\sqrt{3}}{3}+\sqrt{2}\right) \|\tilde{\mathbb{S}}_2f\|_{L^2}
\|\mathcal{H}^{\frac{\alpha+1}{2}}\tilde{\mathbb{S}}_{N-2} g\|_{L^2}
\|\mathcal{H}^{\frac{\alpha+1}{2}}\tilde{\mathbb{S}}_N h\|_{L^2}.
\end{align*}
This is the first result of Proposition \ref{estimatetri}.

For the second inequality of the Proposition \ref{estimatetri}, we just to use,
$$
e^{ct (2n+l+\frac{3}{2})}=e^{ct (2(n-2)+(l+2)+\frac{3}{2})}e^{2ct}=e^{ct (2(n-1)+l+\frac{3}{2})}e^{2ct}
=e^{ct (2n+(l-2)+\frac{3}{2})}e^{2ct}.
$$
This ends the proof of Proposition \ref{estimatetri}.
\end{proof}


\section{The convergence of the formal solution}\label{S4}

In this section, we study the convergence of the solutions $\{g_N; N\in\mathbb{N}\}$ defined by \eqref{gN} in Theorem \ref{th:ODE} where the initial data is the sequence $\{g^0_{n, l, m}=\langle g_0, \varphi_{n, l, m}\rangle; n, l\in\mathbb{N}, |m|\le l\}$ with $g_0\in\,Q^{\alpha}\cap \mathcal{N}^{\perp}$.  Note that  $g_N\in\mathscr{S}(\mathbb{R}^3)\cap\mathcal{N}^{\perp},$
from the definition of \eqref{SN} and \eqref{SN2}, we have
\begin{equation}\label{S}
\mathbb{S}_Ng_N(t)=\tilde{\mathbb{S}}_Ng_N(t)=g_N(t) \in\mathscr{S}(\mathbb{R}^3)\cap\mathcal{N}^{\perp}.
\end{equation}
In particular, for $N=2$, we have
\begin{equation}\label{estimates-2}
\|g_2(t)\|_{L^2(\mathbb{R}^3)}=\|\tilde{\mathbb{S}}_2g_2(t)\|_{L^2(\mathbb{R}^3)}\leq\,e^{-12t}\|\tilde{\mathbb{S}}_2g_0\|_{L^2(\mathbb{R}^3)}.
\end{equation}
Moreover, we recall the result \eqref{L-perp} that
\begin{equation*}
\mathbb{S}_N{\bf L}(g_N(t), g_N(t))=\tilde{\mathbb{S}}_N{\bf L}(g_N(t), g_N(t)).
\end{equation*}
Therefore, we can rewrite the Cauchy problem \eqref{eq-2} as follows:
\begin{equation}\label{eq-3}
\left\{ \begin{aligned}
         &\partial_t g_N+\mathcal{L}(g_N)=\tilde{\mathbb{S}}_{N}{\bf L}(g_N, g_N),\,\\
         &g_{N}|_{t=0}=\sum_{\substack{2\leq2n+l\leq\,N\\n+l\ge2}} \sum_{|m|\le\,l}\langle g_0, \varphi_{n, l, m}\rangle \varphi_{n,l,m}.
\end{aligned} \right.
\end{equation}
Now for $N>2$, $c>0$, taking the inner product of $e^{2ct\mathcal{H}}\mathcal{H}^{\alpha}g_N(t)$ in $L^2(\mathbb{R}^3)$ on both sides of \eqref{eq-3}, we have
\begin{align*}
&\left(\partial_t g_N(t),e^{2ct\mathcal{H}}\mathcal{H}^{\alpha}g_N(t)\right)_{L^2(\mathbb{R}^3)}
+\left(\mathcal{L}g_N(t),e^{2ct\mathcal{H}}\mathcal{H}^{\alpha}g_N(t)\right)_{L^2(\mathbb{R}^3)}\\
&=\left(\tilde{\mathbb{S}}_N\mathbf{L}(g_N,g_N), e^{2ct\mathcal{H}}\mathcal{H}^{\alpha}g_N(t)\right)_{L^2(\mathbb{R}^3)},
\end{align*}
where $g_N$ is defined in \eqref{gN}. Since
\begin{align*}
&\lambda_{0,2}=12>\frac{16}{11}\left(2+\frac{3}{2}\right),\\
&\lambda_{n,l}=2(2n+l)+l(l+1)\geq \frac{16}{11}\left(2n+l+\frac{3}{2}\right), \quad \forall \,2n+l>2.
\end{align*}
The orthogonality of the basis $\{\varphi_{n,l,m}\}_{\{n,l\in \mathbb{N},\,m\in\mathbb{Z},|m|\leq\,l\}}$ imply that
$$
\Big(\mathcal{L}g_N(t),e^{2ct\mathcal{H}}\mathcal{H}^{\alpha} g_N(t)\Big)_{L^2(\mathbb{R}^3)}\geq
\frac{16}{11}\|e^{ct \mathcal{H}}\mathcal{H}^{\frac{\alpha+1}{2}}g_N(t)\|^2_{L^2(\mathbb{R}^3)}\,.
$$
On the other hand
\begin{align*}
&2\left(\partial_t g_N(t),e^{2ct\mathcal{H}}\mathcal{H}^{\alpha}g_N(t)\right)_{L^2(\mathbb{R}^3)}+2 c\left(g_N (t), \mathcal{H}\, e^{2ct\mathcal{H}}\mathcal{H}^{\alpha}g_N(t)\right)_{L^2(\mathbb{R}^3)}\\
&=\frac{d}{dt}\left(g_N(t), e^{2ct\mathcal{H}}\mathcal{H}^{\alpha}g_N(t)\right)_{L^2(\mathbb{R}^3)}=\frac{d}{dt}\|e^{ct \mathcal{H}}\mathcal{H}^{\frac{\alpha}{2}}g_N(t)\|^2_{L^2(\mathbb{R}^3)}\,.
\end{align*}
Therefore, we have
\begin{align*}
&\frac{1}{2}\frac{d}{dt}\|e^{ct \mathcal{H}}\mathcal{H}^{\frac{\alpha}{2}}g_N(t)\|^2_{L^2(\mathbb{R}^3)}
+\left(\frac{16}{11}- c\right)\|e^{ct \mathcal{H}}\mathcal{H}^{\frac{\alpha+1}{2}}g_N(t)\|^2_{L^2(\mathbb{R}^3)}\\
&= \Big(\mathbf{L}(\tilde{\mathbb{S}}_Ng_N,\tilde{\mathbb{S}}_Ng_N), e^{2ct\mathcal{H}}\mathcal{H}^{\alpha}\tilde{\mathbb{S}}_Ng_N(t)\Big)_{L^2}.
\end{align*}
It follows from Proposition \ref{estimatetri} and the inequality \eqref{estimates-2}  that, for any $N> 2$, $t> 0$,
\begin{align}\label{ele}
&\frac{1}{2}\frac{d}{dt}\|e^{ct\mathcal{H}}g_N(t)\|^2_{Q^\alpha(\mathbb{R}^3)}
+\left(\frac{16}{11}- c\right)\|e^{ct\mathcal{H}}g_N(t)\|^2_{Q^{\alpha+1}(\mathbb{R}^3)}\nonumber\\
\leq& \left(\frac{4\sqrt{3}}{3}+\sqrt{2}\right) e^{-(12-c)t}\|\tilde{\mathbb{S}}_2g_0\|_{L^2(\mathbb{R}^3)}
\|e^{ct\mathcal{H}}g_{N-2}\|_{Q^{\alpha+1}(\mathbb{R}^3)}
\|e^{ct\mathcal{H}}g_N\|_{Q^{\alpha+1}(\mathbb{R}^3)}\nonumber\\
\leq& \left(\frac{4\sqrt{3}}{3}+\sqrt{2}\right) e^{-(12-c)t}\|\tilde{\mathbb{S}}_2g_0\|_{L^2(\mathbb{R}^3)}
\|e^{ct\mathcal{H}}g_N\|^2_{Q^{\alpha+1}(\mathbb{R}^3)}
\end{align}
where we used the definition of the shubin spaces $Q^{\alpha+1}(\mathbb{R}^3)$ that
$$\|e^{ct\mathcal{H}}g_N\|^2_{Q^{\alpha+1}(\mathbb{R}^3)}
=\sum_{\substack{2\leq2n+l\leq\,N\\n+l\geq2}}\sum_{|m|\leq\,l}
e^{2ct(2n+l+\frac{3}{2})}(2n+l+\frac{3}{2})^{\alpha+1}
|g_{n,l,m}|^2.$$

\begin{proposition}\label{induction}
There exists $c_0>0, c_1>0$ such that for all $g_0\in Q^{\alpha}(\mathbb{R}^3) \cap \mathcal{N}^{\perp} $ with $\alpha\leq0$, and
$$
\|\tilde{\mathbb{S}}_2g_0\|_{L^2(\mathbb{R}^3)}
=\left(\sum_{|m_2|\leq2}|\langle g_0, \varphi_{0,2,m_2}\rangle|^2\right)^{\frac{1}{2}}\leq c_0\, ,
$$
if $\{g_{n,l,m}(t); n,l\in\mathbb{N},\,m\in\mathbb{Z},\,|m|\leq\,l\}$ is the solution of \eqref{ODE-1} with inital datum $\{g^0_{n,l,m}=\langle g_0, \varphi _{n, l, m}\rangle; n,l\in\mathbb{N},\,|m|\leq\,l\}$, then, for any $N\ge 2, t>0$,
\begin{align}\label{est-1}
\|e^{c_1t\mathcal{H}}g_N(t)\|^2_{Q^\alpha(\mathbb{R}^3)}+c_1\int^t_0
\|e^{c_1\tau\mathcal{H}}g_N(\tau )\|^2_{Q^{\alpha+1}(\mathbb{R}^3)}d\tau
\le \|g_0\|^2_{Q^{\alpha}(\mathbb{R}^3)}.
\end{align}
We have also, for any $t\geq0$ and any $N\geq 2$,
\begin{equation}\label{est-2}
\|\tilde{\mathbb{S}}_N {\bf L}(g_N(t), g_N(t))\|_{Q^{\alpha-2}(\mathbb{R}^3)}\le 2\|g_0\|_{Q^{\alpha}(\mathbb{R}^3)}.
\end{equation}
\end{proposition}

\begin{remark}\label{c-0-1}
It is enough to take $0<c_1<<1$ very small  such that
$$
0<c_0=\frac{\frac{16}{11}-\frac 32 c_1}{\frac{4\sqrt{3}}{3}+\sqrt{2}}<\frac{\frac{16}{11}}{\frac{4\sqrt{3}}{3}+\sqrt{2}}\approx0.39.
$$
\end{remark}

\begin{proof}
For $N=2$,  it follows from Proposition \ref{SSN2} and $\alpha\le 0$ that
\begin{align*}
&\|e^{c_1t\mathcal{H}}\mathcal{H}^{\frac{\alpha}{2}}g_2(t)\|^2_{L^2(\mathbb{R}^3)}+
c_1\int^t_0
\|e^{c_1\tau\mathcal{H}}\mathcal{H}^{\frac{\alpha+1}{2}}g_2(\tau )\|^2_{L^2(\mathbb{R}^3)}d\tau\\
&=\frac{48-21c_1}{48-14c_1}e^{7c_1t-24t}\sum_{|m|\leq2}\Big(\frac{7}{2}\Big)^{\alpha}|g^0_{0,2,m}|^2\\
&\leq\sum_{|m|\leq2}\Big(\frac{7}{2}\Big)^{\alpha}|g^0_{0,2,m}|^2\leq\|g_0\|^2_{Q^{\alpha}(\mathbb{R}^3)}.
\end{align*}
Then for $N>2$, we can deduce from \eqref{ele} with $c=c_1$ and the hypothesis $\|\tilde{\mathbb{S}}_2g_0\|_{L^2(\mathbb{R}^3)}\le c_0 $  that
$$
\frac{d}{dt}\|e^{c_1 t\mathcal{H}}\mathcal{H}^{\frac{\alpha}{2}}g_N(t)\|^2_{L^2(\mathbb{R}^3)}
+c_1
\|e^{c_1t\mathcal{H}}g_N(t)\|^2_{Q^{\alpha+1}(\mathbb{R}^3)}\leq0.
$$
This ends the proof of \eqref{est-1} by integration over $[0, t]$.

Now we prove the estimate \eqref{est-2}.
For $h\in Q^{-\alpha+2}(\mathbb{R}^3)$ with $\alpha\leq0$, by  using the first inequality in the Proposition \ref{estimatetri} with $\alpha$ replacing by $\alpha-1$ where $\alpha-1\leq-1<0$, and the notation of \eqref{S}
\begin{align*}
&|\langle\tilde{\mathbb{S}}_N\mathbf{L}(g_N, \, g_N), h\rangle|\\
=&|(\mathbf{L}(\tilde{\mathbb{S}}_N g_N, \, \tilde{\mathbb{S}}_N g_N),\mathcal{H}^{\alpha-1}\tilde{\mathbb{S}}_N \mathcal{H}^{1-\alpha} h)_{L^2(\mathbb{R}^3)}|
\\
\le & \Big(\frac{4\sqrt{3}}{3}+\sqrt{2}\Big) \|\tilde{\mathbb{S}}_2g_N\|_{L^2(\mathbb{R}^3)}
\|\mathcal{H}^{\frac{\alpha}{2}}\tilde{\mathbb{S}}_{N-2} g_N\|_{L^2(\mathbb{R}^3)}
\|\tilde{\mathbb{S}}_N h\|_{Q^{-\alpha+2}(\mathbb{R}^3)}\\
\le & \Big(\frac{4\sqrt{3}}{3}+\sqrt{2}\Big) \|g_2\|_{L^2(\mathbb{R}^3)}
\|\mathcal{H}^{\frac{\alpha}{2}}g_{N-2}\|_{L^2(\mathbb{R}^3)}
\| h\|_{Q^{-\alpha+2}(\mathbb{R}^3)}.
\end{align*}
Then by the definition of the norm and \eqref{est-1}, we have
\begin{align*}
&\|\tilde{\mathbb{S}}_N{\bf L}(g_N(t), g_N(t))\|_{Q^{\alpha-2}(\mathbb{R}^3)}\\
=&\sup_{\|h\|_{Q^{2-\alpha}(\mathbb{R}^3)}=1}|\langle\mathbb{S}_N\mathbf{L}( g_N, \, g_N),h\rangle|\\
\leq& \Big(\frac{4\sqrt{3}}{3}+\sqrt{2}\Big)  \|\tilde{\mathbb{S}}_2g_0\|_{L^2}
\|\mathcal{H}^{\frac{\alpha}{2}}g_{N-2}\|_{L^2(\mathbb{R}^3)}\leq \frac{16}{11}\|\mathcal{H}^{\frac{\alpha}{2}}g_0\|_{L^2(\mathbb{R}^3)}
\end{align*}
which ends the proof of the Proposition \ref{induction}.
\end{proof}

In particulary, we get the following surprise results
\begin{corollary}\label{cor1}
For any $f, g \in Q^{\alpha}(\mathbb{R}^3)\cap\mathcal{N}^{\perp}$ with $\alpha\leq0$, we have
$$
{\bf L}(f,\, g)\in Q^{\alpha-2}(\mathbb{R}^3)\cap\mathcal{N}^{\perp}	
$$
and
\begin{equation*}
\|{\bf L}(f, \, g)\|_{Q^{\alpha-2}(\mathbb{R}^3)}\le \Big(\frac{4\sqrt{3}}{3}+\sqrt{2}\Big) \|\tilde{\mathbb{S}}_2f\|_{L^2(\mathbb{R}^3)}
\|g\|_{Q^{\alpha}(\mathbb{R}^3)}.
\end{equation*}
\end{corollary}

\smallskip
\noindent
{\bf Convergence in Shubin space.}\,\, We prove now the convergence of the sequence
$$
g_N(t)\rightarrow
\,g(t)=\sum^{+\infty}_{k=2}\sum_{\substack{2n+l=k\\n+l\ge2}} \sum_{|m|\le\,l}g_{n,l,m}(t)\varphi_{n,l,m}
$$
where for all $N\geq2$, $g_N$ was defined in \eqref{gN} with the coefficients  $\{g_{n,l,m}(t)\}$ defined in \eqref{ODE-1},  
the inital datum is $\{g^0_{n,l,m}=\langle g_0, \varphi _{n, l, m}\rangle; n,l\in\mathbb{N},\,|m|\leq\,l\}$ with $g_0\in Q^\alpha(\mathbb{R}^3)
\cap\mathcal{N}^{\perp}$.
By Proposition \ref{induction} and the orthogonality of the basis $\{\varphi_{n,l,m}\}$,
\begin{align*}
\sum_{\substack{2\leq2n+l\leq\,N\\n+l\ge2}} \sum_{|m|\le\,l}
e^{2c_1t(2n+l+\frac{3}{2})}(2n+l+\frac{3}{2})^{\alpha}|g_{n,l,m}(t)|^2
\le\|g_0\|^2_{Q^{\alpha}(\mathbb{R}^3)}.
\end{align*}
It follows that for all $t\geq0$,
\begin{align*}
&\|g_{N}(t)\|^2_{Q^{\alpha}(\mathbb{R}^3)}=\|\mathcal{H}^{\frac{\alpha}{2}}g_{N}(t)\|^2_{L^2(\mathbb{R}^3)}\notag\\
&=
\sum_{\substack{2\leq2n+l\leq\,N\\n+l\ge2}} \sum_{|m|\le\,l}
(2n+l+\frac{3}{2})^{\alpha}|g_{n,l,m}(t)|^2
\le\|g_0\|^2_{Q^{\alpha}(\mathbb{R}^3)}.
\end{align*}
By using the monotone convergence theorem, we have
$$
g_{N}\,\, \to\,\, g(t)\,\in\, Q^{\alpha}(\mathbb{R}^3).
$$
Moreover, for any $T> 0$,
\begin{equation*}
\lim_{N\to \infty}\|g_{N}-g\|_{L^\infty([0, T]; Q^{\alpha}(\mathbb{R}^3))}=0\, .
\end{equation*}
On the other hand, using \eqref{est-2} and Corollary \ref{cor1} , we have also
\begin{equation*}
\tilde{\mathbb{S}}_N{\bf L}(g_{N}, g_{N})
\,\,\to\,\, {\bf L}(g,  g),
\end{equation*}
in $Q^{\alpha-2}(\mathbb{R}^3)$.

\smallskip
We recall the definition of weak solution of \eqref{eq-1}:

\begin{definition}
Let $g_0\in \mathcal{S}'(\mathbb{R}^3)$,  $g(t, v)$ is called a weak solution of the Cauchy problem \eqref{eq-1} if it satisfies the following conditions:
\begin{align*}
&g\in C^0([0, +\infty[; \mathcal{S}'(\mathbb{R}^3)), \quad g(0, v)=g_0(v),\\
&\mathcal{L} ( g )\in L^2([0,  T[; \, \mathcal{S}'(\mathbb{R}^3)),\quad
{\bf L} (g , g)\in L^2([0,  T[; \mathcal{S}'(\mathbb{R}^3)),\quad\forall T>0,\\
&\langle g(t), \phi(t)\rangle-\langle g_0, \phi(0)\rangle+\int^t_{0}\langle\mathcal{L}g(\tau), \phi(\tau)\rangle d\tau\\
&\qquad\qquad=\int^t_{0}\langle\,g(\tau), \partial_{\tau}\phi(\tau)\rangle\,d\tau +\int^t_{0}\langle{\bf L}(g(\tau),g(\tau)), \phi(\tau)\rangle d\tau,\quad \forall t\ge 0,	
\end{align*}
For any $\phi(t)\in\,C^1\big([0, +\infty[; \mathscr{S}(\mathbb{R}^3)\big)$.
\end{definition}

We prove now the main Theorem \ref{trick}.

\smallskip\noindent
{\bf Existence of weak solution.}

Let $\{ g_{n,l,m}, n,l\in\mathbb{N}, n+l\ge 2, |m|\leq\,l\}$ be the solution of the infinite system \eqref{ODE-1}  in Theorem  \ref{th:ODE}  with the initial datum give in the Proposition \ref{induction}, then for any $N\ge 2$, $g_N$ satisfy the equation \eqref{eq-3}.

We have, firstly, from the Proposition \ref{induction},  there exists positive constant $C>0$, for any $N\ge 2$ and any $T>0$,
\begin{align*}
&\|g_{N}\|_{L^\infty([0, T];  Q^{\alpha}(\mathbb{R}^3))}
\le \|g_0\|_{Q^{\alpha}(\mathbb{R}^3)}\, ,\\
&\|\mathcal{L} ( g_{N} )\|_{L^2([0, T];  Q^{\alpha-3}(\mathbb{R}^3))}
\le C\|g_0\|_{Q^{\alpha}(\mathbb{R}^3)}\, ,\\
&\|\tilde{\mathbb{S}}_{N}{\bf L}(g_{N}, g_{N})\|_{L^2([0, T];  Q^{\alpha-2}(\mathbb{R}^3))}
\le C\|g_0\|_{Q^{\alpha}(\mathbb{R}^3)}.	
\end{align*}
So that  the equation \eqref{eq-3} implies that the sequence $\{\frac{d}{dt} \tilde{\mathbb{S}}_N g(t)\}$ is uniformly bounded  in $  Q^{\alpha-3}(\mathbb{R}^3)$ with respect to $N\in\mathbb{N}$ and $t\in [0, T]$.  The Arzel$\grave{a}$-Ascoli Theorem implies that
$$
g_{N} \,\, \to\,\, g\in C^0([0, +\infty[;   Q^{\alpha-2}(\mathbb{R}^3))
 \subset C^0([0, +\infty[; \mathcal{S}'(\mathbb{R}^3)),
$$
and
$$
g(0)=g_0.
$$
Secondly,  for any $\phi(t)\in\,C^1\Big(\mathbb{R}_+,\mathscr{S}(\mathbb{R}^3)\Big)$, the Cauchy problem \eqref{eq-3} can be rewrite as follows
\begin{align*}
&\langle g_{N}(t), \phi(t)\rangle
-\langle g_{N}(0), \phi(0)\rangle-\int^t_{0}\langle\,g_{N}(\tau), \partial_{\tau}\phi(\tau)\rangle\,d\tau \\
&=-\int^t_{0}\langle\mathcal{L}g_{N}(\tau), \phi(\tau)\rangle d\tau+\int^t_{0}\langle\tilde{\mathbb{S}}_{N}{\bf L}(g_{N}(\tau),g_{N}(\tau)), \phi(\tau)\rangle d\tau
\end{align*}
Let $N\rightarrow+\infty$, we conclude that,
\begin{align*}
&\langle g(t), \phi(t)\rangle-\langle g_0, \phi(0)\rangle-\int^t_{0}\langle\,g(\tau), \partial_{\tau}\phi(\tau)\rangle\,d\tau \\
&=-\int^t_{0}\langle\mathcal{L}(g(\tau)), \phi(\tau)\rangle d\tau+\int^t_{0}\langle{\bf L}(g(\tau),g(\tau)), \phi(\tau)\rangle d\tau,
\end{align*}
which shows $g\in L^\infty([0, +\infty[; Q^{\alpha}(\mathbb{R}^3))$ is a global weak solution of Cauchy problem \eqref{eq-1}.

\smallskip\noindent
{\bf Regularity of the solution.}\,\, For $N\ge 2$,
we deduce from the formulas  \eqref{est-1} and the orthogonality of the basis $(\varphi_{n, l, m})$ that
\begin{align*}
\|e^{c_1t\mathcal{H}}\mathcal{H}^{\frac{\alpha}{2}}g_{N}(t)\|_{L^2(\mathbb{R}^3)}\leq
\|\mathcal{H}^{\frac{\alpha}{2}}g_0\|_{L^2(\mathbb{R}^3)},\quad\forall\, N\ge 2,\,\,\,\, t\ge 0,
\end{align*}
by using the monotone convergence theorem, we conclude that, such that
\begin{equation*}
\|e^{c_1 t\mathcal{H}}\mathcal{H}^{\frac{\alpha}{2}}g(t)\|_{L^2(\mathbb{R}^3)}\leq
\|g_0\|_{Q^{\alpha}(\mathbb{R}^3)},\quad \forall\,\,t\ge 0.
\end{equation*}
The proof of Theorem \ref{trick} is completed.


\section{The techincal computations}\label{S5}

The proof of the main technic part was presented in this section.  More precisely, we prepare to prove Proposition \ref{expansion} in Section 2 and Proposition \ref{recur} in Section 3.

To this ends, we need to state some Lemmas and new notations. Recall firstly
\begin{equation}\label{relation2}
v_kv_j\sqrt{\mu}\in\text{span}\left\{\varphi_{0,2,0},\,\varphi_{0,2,\pm1},\,\varphi_{0,2,\pm2},\,\varphi_{1,0,0},\,
\varphi_{0,0,0}\right\}.
\end{equation}
This relation is important in the expansion of the nonlinear operators.

Setting
\begin{equation*}
\Psi_{n,l,m}(v)=\sqrt{\mu}(v)\varphi_{n,l,m}(v).
\end{equation*}
Recalled Lemma 7.2 of \cite{GLX} that the Fourier transformation is ,
\begin{equation}\label{fourier}
\widehat{\Psi_{n,l,m}}(\xi)=B_{n,l}|\xi|^{2n+l}e^{-\frac{|\xi|^2}{2}}Y^m_l(\frac{\xi}{|\xi|}),
\end{equation}
where
\begin{align}\label{B_NL}
B_{n,l}=(-i)^l(2\pi)^{\frac{3}{4}}
\left(\frac{1}{\sqrt{2}n!\Gamma(n+l+\frac{3}{2})2^{2n+l}}\right)^{\frac{1}{2}}.
\end{align}

For the Laplace-Beltrami operator on the unit sphere $\mathbb{S}^2$, see also Section 4.2 in \cite{NYKC2}, we have
\begin{equation}\label{Laplace-Beltrami}
\sum_{\substack{1\leq\,k,j\leq3\\k\neq\,j}}\left(v^2_j\partial^2_{v_k}-v_kv_j\partial_{v_j}\partial_{v_k}\right)
-2\sum^3_{k=1}v_k\partial_{v_k}=\Delta_{\mathbb{S}^2}.
\end{equation}
And for $n,l\in\mathbb{N}$, $m\in\mathbb{Z}$, $|m|\leq\,l$,
 \begin{equation}\label{partial_2}
\left[\sum^3_{k=1}\partial^2_{v_k}+v_k\partial_{v_k}\right]\Psi_{n,l,m}=-(2n+l+3)\Psi_{n,l,m}.
\end{equation}

Recall that the family $\Big(Y^m_l(\sigma)\Big)_{l\geq0,|m|\leq\,l}$ is the orthonormal basis of
$L^2(\mathbb{S}^2,\,d\sigma)$ (see $(16)$ of Chap.1 in \cite{Jones}). We have the following addition lemma,

\begin{lemma}\label{harmonic}
For any $\omega\in\mathbb{S}^2$, $l,\tilde{l}\in\mathbb{N}$, $|m|\leq\,l$, $|\tilde{m}|\leq\,\tilde{l}$,
$$Y^m_l(\omega)Y^{\tilde{m}}_{\tilde{l}}(\omega)
=\sum_{0\leq\,p\leq\min(l,\tilde{l})}\left(\int_{\mathbb{S}^2}Y^m_l(\omega)
Y^{\tilde{m}}_{\tilde{l}}(\omega)Y^{-m-\tilde{m}}_{l+\tilde{l}-2p}(\omega)d\omega\right)Y^{m+\tilde{m}}_{l+\tilde{l}-2p}(\omega)$$
where we always define $Y^{-m-\tilde{m}}_{l+\tilde{l}-2p}(\omega)\equiv0$, if $|m+\tilde{m}|>l+\tilde{l}-2p.$
\end{lemma}

\begin{proof}
For the proof of Lemma \ref{harmonic}, we refer to (86) in Chap.\,3 of \cite{Jones} or  Gaunt's formula from (13-12) of Chap.13 in \cite{JCSlater}.
\end{proof}

In particular, for $l=1$ or $l=2$ in Lemma \ref{harmonic}, we have

\begin{corollary}\label{harmonic-c}
For all $\omega\in\mathbb{S}^2$, $l\in\mathbb{N}$, $|m|\leq\,l$, $|m_1|\leq\,1$, $|m_2|\leq\,2$,
\begin{align*}
&Y^{m_1}_1(\omega)Y^{m}_{l}(\omega)
=\sum_{0\leq\,p\leq\min(1,l)}\tilde{C}^{m_1,m}_{l,l+1-2p}
Y^{m_1+m}_{l+1-2p}(\omega);\\
&Y^{m_2}_2(\omega)Y^{m}_{l}(\omega)
=\sum_{0\leq\,p\leq\min(2,l)}C^{m_2,m}_{l,l+2-2p}
Y^{m_2+m}_{l+2-2p}(\omega)
\end{align*}
where
$$
\tilde{C}^{m_1,m}_{l,l+1-2p}=\int_{\mathbb{S}^2}Y^{m_1}_1(\omega)
Y^{m}_{l}(\omega)Y^{-m_1-m}_{l+1-2p}(\omega)d\omega,
$$
$$
C^{m_2,m}_{l,l+2-2p}=\int_{\mathbb{S}^2}Y^{m_2}_2(\omega)
Y^{m}_{l}(\omega)Y^{-m_2-m}_{l+2-2p}(\omega)d\omega.
$$
More explicitly, for any $l\in\mathbb{N}, |m|\leq\,l$
$$
Y^{m_1}_1(\omega)Y^{m}_{l}(\omega)=\tilde{C}^{m_1,m}_{l,l+1}
Y^{m_1+m}_{l+1}(\omega)+\tilde{C}^{m_1,m}_{l,l-1}
Y^{m_1+m}_{l-1}(\omega),
$$
and
\begin{align*}
&Y^{m_2}_2(\omega)Y^{m}_{l}(\omega)\\
&=C^{m_2,m}_{l,l+2}Y^{m_2+m}_{l+2}(\omega)
+C^{m_2,m}_{l,l}Y^{m_2+m}_{l}(\omega)
+C^{m_2,m}_{l,l-2}
Y^{m_2+m}_{l-2}(\omega)
\end{align*}
where, for convenience, we note
$$
Y^{m_1+m}_{-1}(\omega)=0,\quad\,Y^{m_2+m}_{-2}(\omega)=0.
$$
\end{corollary}

\begin{lemma}\label{addition}
Let $\Psi_{n,l,m}=\sqrt{\mu}\varphi_{n,l,m}$, then for $v\in\mathbb{R}^3$, we have\\
1) For $|m_1|\leq1$,
 \begin{equation}\label{orth-1}
\int_{\mathbb{R}^3_{v_*}}(v\cdot\,v_*)\Psi_{0,1,m_1}(v_*)dv_*=\sqrt{\frac{4\pi}{3}}|v|Y^{m_1}_1(\sigma).
\end{equation}
2) For $|m_2|\leq2$,
\begin{equation}\label{orth-2}
\int_{\mathbb{R}^3_{v_*}}(v\cdot\,v_*)^2\Psi_{0,2,m_2}(v_*)dv_*=\sqrt{\frac{16\pi}{15}}|v|^2Y^{m_2}_2(\sigma).
\end{equation}
3) and
\begin{equation}\label{orth-3}
\int_{\mathbb{R}^3_{v_*}}(v\cdot\,v_*)^2\Psi_{1,0,0}(v_*)dv_*=-\frac{\sqrt{6}}{3}|v|^2\,.
\end{equation}
\end{lemma}

\begin{proof}
Set $\sigma*=\frac{v_*}{|v_*|}, \sigma=\frac{v}{|v|}\in\mathbb{S}^2$, then
\begin{align*}
&\int_{\mathbb{R}^3_{v_*}}(v\cdot\,v_*)\Psi_{0,1,m_1}(v_*)dv_*=\int_{\mathbb{R}^3_{v_*}}|v||v_*|(\sigma\cdot\,\sigma*)\Psi_{0,1,m_1}(v_*)dv_*\\
&\int_{\mathbb{R}^3_{v_*}}(v\cdot\,v_*)^2\Psi_{0,2,m_2}(v_*)dv_*=\int_{\mathbb{R}^3_{v_*}}|v|^2|v_*|^2(\sigma\cdot\,\sigma*)^2\Psi_{0,2,m_2}(v_*)dv_*\\
&\int_{\mathbb{R}^3_{v_*}}(v\cdot\,v_*)^2\Psi_{1,0,0}(v_*)dv_*=\int_{\mathbb{R}^3_{v_*}}|v|^2|v_*|^2(\sigma\cdot\,\sigma*)^2\Psi_{1,0,0}(v_*)dv_*
\end{align*}
By using the formulas $(5_3)$ in Sec.1, Chap. III of \cite{San} that
\begin{align*}
P_1(x)=x,\quad P_2(x)=\frac{3}{2}x^2-\frac{1}{2}.
 \end{align*}
We apply the addition theorem of spherical harmonics $(7-34)$ in Chapter 7 of \cite{JCSlater} (see also (VII) in Sec.19, Chapter III of \cite{San}) that,
$$P_k(\sigma_*\cdot\sigma)=\frac{4\pi}{2k+1}\sum_{|m_k|\leq\,k}Y^{-m_k}_k(\sigma_*)Y^{m_k}_k(\sigma).$$
Then
\begin{align*}
&\sigma_*\cdot\sigma=P_1(\sigma_*\cdot\sigma)=\frac{4\pi}{3}\sum_{|\tilde{m}_1|\leq1}Y^{-\tilde{m}_1}_1(\sigma_*)Y^{\tilde{m}_1}_1(\sigma);\\
&(\sigma_*\cdot\sigma)^2=\frac{2}{3}P_2(\sigma_*\cdot\sigma)+\frac{1}{3}=\frac{8\pi}{15}\sum_{|\tilde{m}_2|\leq2}Y^{-\tilde{m}_2}_2(\sigma_*)Y^{\tilde{m}_2}_2(\sigma)+\frac{1}{3}.
\end{align*}
Substituting this expansion into the integral and using the orthogonal of the eigenfunctions $\varphi_{n,l,m}$ in $L^2(\mathbb{R}^3)$, one has
\begin{align*}
&\int_{\mathbb{R}^3_{v_*}}|v||v_*|(\sigma\cdot\,\sigma*)\Psi_{0,1,m_1}(v_*)dv_*=\sqrt{\frac{4\pi}{3}}\, |v|\, Y^{m_1}_1(\sigma)\\
&\int_{\mathbb{R}^3_{v_*}}|v|^2|v_*|^2(\sigma\cdot\,\sigma*)^2\Psi_{0,2,m_2}(v_*)dv_*=\sqrt{\frac{16\pi}{15}}\, |v|^2Y^{m_2}_2(\sigma)\\
&\int_{\mathbb{R}^3_{v_*}}(v\cdot\,v_*)^2\Psi_{1,0,0}(v_*)dv_*=\frac{1}{3}\left(\int_{\mathbb{R}^3_{v_*}}|v_*|^2\Psi_{1,0,0}(v_*)dv_*\right)|v|^2
=-\frac{\sqrt{6}}{3}\, |v|^2,
\end{align*}
where we used the explicit formula of $\varphi_{0,1,m_1}$,  $\varphi_{1,0,0}$ and $\varphi_{0,2,m_2}$ in \eqref{S22} in section \ref{S2}.
This ends the proof of Lemma \ref{addition}.
\end{proof}

We prove now the 5 parts of the  Proposition of \ref{expansion} by the following 5 Lemmas
\begin{lemma}\label{lemma5.1+}
For $n,l\in\mathbb{N}$, $|m|\leq\,l$, we have
\begin{align*}
{\bf L}(\varphi_{0,0,0},\varphi_{n,l,m})=-\left(2(2n+l)+l(l+1)\right)\varphi_{n,l,m}.
\end{align*}
\end{lemma}

\begin{proof}
Since $\varphi_{0,0,0}=\sqrt{\mu},$
then for all $n,l\in\mathbb{N},|m|\le\,l$, one has
\begin{align*}
&{\bf L}(\varphi_{0,0,0},\, \varphi_{n,l,m})
=\frac{1}{\sqrt{\mu(v)}}Q(\mu,\Psi_{n,l,m})\\
=&\frac{1}{\sqrt{\mu(v)}}\sum_{1\leq\,k,j\leq3}\partial_{v_k}\int_{\mathbb{R}^{3}}
a_{k,j}(v-v_*)\left[\mu(v_*)
\partial_{v_j}\Psi_{n,l,m}(v)
-\partial_{v^*_j}\mu(v_*)\Psi_{n,l,m}(v)\right]
dv_*
\end{align*}
where
\begin{align*}
a_{k,k}(v-v_*)&=\sum_{\substack{1\leq\,j\leq3\\j\neq\,k}}(v_j-v_j^*)^2;\\
a_{k,j}(v-v_*)&=-(v_k-v_k^*)(v_j-v_j^*)\,\text{when}\,\,k\neq\,j.
\end{align*}
It follows that
$${\bf L}(\varphi_{0,0,0},\varphi_{n,l,m})={\bf L}_1(\varphi_{0,0,0},\varphi_{n,l,m})-{\bf L}_2(\varphi_{0,0,0},\varphi_{n,l,m})$$
where
\begin{align*}
&{\bf L}_1(\varphi_{0,0,0},\varphi_{n,l,m})\\
&=
\frac{1}{\sqrt{\mu(v)}}\sum_{\substack{1\leq\,k,j\leq3\\k\neq\,j}}
\partial_{v_k}\left(\int_{\mathbb{R}^{3}}\left(v^2_j-2v_jv^*_j+(v_j^*)^2\right)
\mu(v_*)dv_*\partial_{v_k}\Psi_{n,l,m}(v)\right)\\
&\quad-\frac{1}{\sqrt{\mu(v)}}\sum_{\substack{1\leq\,k,j\leq3\\k\neq\,j}}
\partial_{v_k}\left(\int_{\mathbb{R}^{3}}\left(v^2_j-2v_jv^*_j+(v_j^*)^2\right)
\partial_{v^*_k}\mu(v_*)dv_*
\Psi_{n,l,m}(v)\right);\\
&{\bf L}_2(\varphi_{0,0,0},\varphi_{n,l,m})\\
&=
\frac{1}{\sqrt{\mu(v)}}\sum_{\substack{1\leq\,k,j\leq3\\k\neq\,j}}
\partial_{v_k}\left(\int_{\mathbb{R}^{3}}\left(v_kv_j-v_k^*v_j-v_kv_j^*+v_k^*v_j^*\right)
\mu(v_*)dv_*
\partial_{v_j}\Psi_{n,l,m}(v)\right)\\
&\,-\frac{1}{\sqrt{\mu(v)}}\sum_{\substack{1\leq\,k,j\leq3\\k\neq\,j}}
\partial_{v_k}\left(\int_{\mathbb{R}^{3}}\left(v_kv_j-v_k^*v_j-v_kv_j^*+v_k^*v_j^*\right)
\partial_{v^*_j}\mu(v_*)dv_*
\Psi_{n,l,m}(v)\right).
\end{align*}
Then by using \eqref{fourier}, we have,
\begin{align*}
&\int_{\mathbb{R}^{3}}
\partial_{v^*_k}\mu(v_*)dv_*
=\left.i\xi_k\widehat{\mu}(\xi)\right|_{\xi=0}=\left.i\xi_k\widehat{\Psi_{0,0,0}}(\xi)\right|_{\xi=0}=0,\quad\,k=1,2,3;\\
&\int_{\mathbb{R}^{3}}
v^*_j\partial_{v^*_k}\mu(v_*)dv_*
=-\left.\xi_k\partial_{\xi^*_j}\widehat{\Psi_{0,0,0}}(\xi)\right|_{\xi=0}=0,\quad\,k\neq\,j;\\
&\int_{\mathbb{R}^{3}}
(v^*_j)^2\partial_{v^*_k}\mu(v_*)dv_*
=-\left.i\xi_k\partial^2_{\xi^*_j}\widehat{\Psi_{0,0,0}}(\xi)\right|_{\xi=0}=0,\quad\,k\neq\,j;\\
&\int_{\mathbb{R}^{3}}
v^*_j\partial_{v^*_j}\mu(v_*)dv_*
=\left.-\widehat{\Psi_{0,0,0}}(\xi)
-\xi_j\partial_{\xi^*_j}\widehat{\Psi_{0,0,0}}(\xi)\right|_{\xi=0}=-1;\\
&\int_{\mathbb{R}^{3}}
v^*_kv^*_j\partial_{v^*_j}\mu(v_*)dv_*
=-\int_{\mathbb{R}^{3}}
v^*_k\mu(v_*)dv_*=0.
\end{align*}
It is obviously that
\begin{align*}
&\int_{\mathbb{R}^{3}}\mu(v_*)dv_*
=(\varphi_{0,0,0},\varphi_{0,0,0})_{L^2(\mathbb{R}^3)}=1,\\
&\int_{\mathbb{R}^{3}}v^*_k\mu(v_*)dv_*=0,\\
&\int_{\mathbb{R}^{3}}v^*_kv^*_j\mu(v_*)dv_*=0,\quad\text{when}\,k\neq\,j.
\end{align*}
Then we can deduce that
\begin{align*}
&{\bf L}(\varphi_{0,0,0},\varphi_{n,l,m})={\bf L}_1(\varphi_{0,0,0},\varphi_{n,l,m})-{\bf L}_2(\varphi_{0,0,0},\varphi_{n,l,m})\\
&=\frac{1}{\sqrt{\mu(v)}}
\left[\sum_{\substack{1\leq\,k,j\leq3\\k\neq\,j}}
\left(v^2_j\partial^2_{v_k}-v_kv_j\partial_{v_j}\partial_{v_k}\right)+6\right]
\Psi_{n,l,m}(v)\\
&\quad+\frac{1}{\sqrt{\mu(v)}}\sum_{\substack{1\leq\,k,j\leq3\\k\neq\,j}}
\left(\int_{\mathbb{R}^{3}}(v_j^*)^2
\mu(v_*)dv_*\right)\partial^2_{v_k}\Psi_{n,l,m}(v)\\
&=\frac{1}{\sqrt{\mu(v)}}
\left[\sum_{\substack{1\leq\,k,j\leq3\\k\neq\,j}}
\left(v^2_j\partial^2_{v_k}-v_kv_j\partial_{v_j}\partial_{v_k}\right)+6\right]
\Psi_{n,l,m}(v)\\
&\quad+\frac{1}{\sqrt{\mu(v)}}
\left(\frac{2}{3}\int_{\mathbb{R}^{3}}|v_*|^2
\mu(v_*)dv_*\right)\sum^3_{k=1}\partial^2_{v_k}\Psi_{n,l,m}(v)\\
&=\frac{1}{\sqrt{\mu(v)}}
\left[\sum_{\substack{1\leq\,k,j\leq3\\k\neq\,j}}
\left(v^2_j\partial^2_{v_k}-v_kv_j\partial_{v_j}\partial_{v_k}\right)+6+2\Delta_v\right]
\Psi_{n,l,m}(v)
\end{align*}
By using \eqref{Laplace-Beltrami} and \eqref{partial_2} that,
\begin{align*}
&\left[\sum_{\substack{1\leq\,k,j\leq3\\k\neq\,j}}
\left(v^2_j\partial^2_{v_k}-v_kv_j\partial_{v_j}\partial_{v_k}\right)+6+2\Delta_v\right]
\Psi_{n,l,m}(v)\\
&=\left[2\sum^3_{k=1}v_k\partial_{v_k}+6+2\Delta_v\right]
\Psi_{n,l,m}(v)+\Delta_{\mathbb{S}^2}\Psi_{n,l,m}(v)\\
&=-2(2n+l)\Psi_{n,l,m}(v)-l(l+1)\Psi_{n,l,m}(v)
\end{align*}
where we used the Laplace-Beltrami operator property
$$
\Delta_{\mathbb{S}^2}\Psi_{n,l,m}=-l(l+1)\Psi_{n,l,m}.
$$
Therefore, we conclude with the notation $\Psi_{n,l,m}=\sqrt{\mu}\varphi_{n,l,m}$ that
$$
{\bf L}(\varphi_{0,0,0},\varphi_{n,l,m})=-\left(2(2n+l)+l(l+1)\right)\varphi_{n,l,m}.
$$
We end the proof of Lemma \ref{lemma5.1+} .
\end{proof}

\begin{lemma}\label{lemma5.2+}
For $n,l\in\mathbb{N}$, $|m|\leq\,l$, we have
\begin{align*}
&{\bf L}(\varphi_{0,1,m_1},\varphi_{n,l,m})\\
=&A^-_{n,l,m,m_1}\varphi_{n+1,l-1,m_1+m}+A^+_{n,l,m,m_1}\varphi_{n,l+1,m_1+m}, \,\forall\,|m_1|\leq1
\end{align*}
with
\begin{equation*}
\begin{split}
&A^-_{n,l,m,m_1}=4\sqrt{\frac{\pi}{3}}(l-1)\sqrt{2(n+1)}
\tilde{C}^{m_1,m}_{l,l-1};\\
&A^+_{n,l,m,m_1}=4\sqrt{\frac{\pi}{3}}(l+2)\sqrt{2n+2l+3}
\tilde{C}^{m_1,m}_{l,l+1}
\end{split}
\end{equation*}
where $\tilde{C}^{m_1,m}_{l,l-1}, \tilde{C}^{m_1,m}_{l,l+1}$
were defined in the Corollary \ref{harmonic-c}.
\end{lemma}

\begin{proof}
For $|m_1|\leq1$, for all $n,l\in\mathbb{N},|m|\le\,l$, one has
\begin{align*}
&{\bf L}(\varphi_{0,1,m_1},\varphi_{n,l,m})
=\frac{1}{\sqrt{\mu(v)}}Q(\Psi_{0,1,m_1},\Psi_{n,l,m})\\
=&\frac{1}{\sqrt{\mu(v)}}\sum_{1\leq\,k,j\leq3}\partial_{v_k}\int_{\mathbb{R}^{3}}
a_{k,j}(v-v_*)\\
&\qquad\times\left[\Psi_{0,1,m_1}(v_*)
\partial_{v_j}\Psi_{n,l,m}(v)
-\partial_{v^*_j}\Psi_{0,1,m_1}(v_*)\Psi_{n,l,m}(v)\right]
dv_*.\\
=&{\bf L}_1(\varphi_{0,1,m_1},\varphi_{n,l,m})-{\bf L}_2(\varphi_{0,1,m_1},\varphi_{n,l,m}).
\end{align*}
Recall,
\begin{align*}
&\int_{\mathbb{R}^{3}}
\partial_{v^*_k}\Psi_{0,1,m_1}(v_*)dv_*
=\left.i\xi_k\widehat{\Psi_{0,1,m_1}}(\xi)\right|_{\xi=0}=0,\quad\,k=1,2,3;\\
&\int_{\mathbb{R}^{3}}
v^*_j\partial_{v^*_k}\Psi_{0,1,m_1}(v_*)dv_*
=-\left.\xi_k\partial_{\xi^*_j}\widehat{\Psi_{0,1,m_1}}(\xi)\right|_{\xi=0}=0,\quad\,k\neq\,j;\\
&\int_{\mathbb{R}^{3}}
(v^*_j)^2\partial_{v^*_k}\Psi_{0,1,m_1}(v_*)dv_*
=-\left.i\xi_k\partial^2_{\xi^*_j}\widehat{\Psi_{0,1,m_1}}(\xi)\right|_{\xi=0}=0,\quad\,k\neq\,j;\\
&\int_{\mathbb{R}^{3}}
v^*_j\partial_{v^*_j}\Psi_{0,1,m_1}(v_*)dv_*
=\left.-\widehat{\Psi_{0,1,m_1}}(\xi)
-\xi_j\partial_{\xi^*_j}\widehat{\Psi_{0,1,m_1}}(\xi)\right|_{\xi=0}=0;\\
&\int_{\mathbb{R}^{3}}
v^*_kv^*_j\partial_{v^*_j}\Psi_{0,1,m_1}(v_*)dv_*
=-\int_{\mathbb{R}^{3}}
v^*_k\Psi_{0,1,m_1}(v_*)dv_*.
\end{align*}
By using the relation \eqref{relation2}, one can verify that
\begin{align*}
&\int_{\mathbb{R}^{3}}\Psi_{0,1,m_1}(v_*)dv_*
=(\varphi_{0,1,m_1},\varphi_{0,0,0})_{L^2(\mathbb{R}^3)}=0,\\
&\int_{\mathbb{R}^{3}}v^*_kv^*_j\Psi_{0,1,m_1}(v_*)dv_*=0,\quad\forall\,1\leq\,k,j\leq3.
\end{align*}
We can conclude that
\begin{align*}
&{\bf L}(\varphi_{0,1,m_1},\varphi_{n,l,m})={\bf L}_1(\varphi_{0,1,m_1},\varphi_{n,l,m})-{\bf L}_2(\varphi_{0,1,m_1},\varphi_{n,l,m})\\
&=\frac{1}{\sqrt{\mu(v)}}\sum_{\substack{1\leq\,k,j\leq3\\k\neq\,j}}\Big[-2\left(\int_{\mathbb{R}^3}v^*_j\Psi_{0,1,m_1}(v_*)dv_*\right)v_j\partial^2_{v_k}\Psi_{n,l,m}(v)\\
&\qquad\qquad\qquad\qquad\quad+\left(\int_{\mathbb{R}^3}v^*_k\Psi_{0,1,m_1}(v_*)dv_*\right)v_j\partial_{v_k}\partial_{v_j}\Psi_{n,l,m}(v)\\
&\qquad\qquad\qquad\qquad\quad+\left(\int_{\mathbb{R}^3}v^*_j\Psi_{0,1,m_1}(v_*)dv_*\right)v_k\partial_{v_k}\partial_{v_j}\Psi_{n,l,m}(v)\Big].
\end{align*}
By Fourier transformation, we have
\begin{align*}
&\mathcal{F}[\sqrt{\mu}\, {\bf L}(\varphi_{0,1,m_1},\varphi_{n,l,m})](\xi)\\
&=2i\sum^3_{j=1}\left(\int_{\mathbb{R}^3}v^*_j\Psi_{0,1,m_1}(v_*)dv_*\right)\partial_{\xi_j}\left(|\xi|^2\widehat{\Psi_{n,l,m}}(\xi)\right)\\
&\quad-2i\left(\int_{\mathbb{R}^3}(v^*\cdot\xi)\Psi_{0,1,m_1}(v_*)dv_*\right)\left(|\xi|\partial_{|\xi|}+4\right)\widehat{\Psi_{n,l,m}}(\xi)\\
&=\widehat{\mathbf{H_{1}}}(\xi)-\widehat{\mathbf{H_{2}}}(\xi)
\end{align*}
where $v^*\cdot\xi=\sum^3_{k=1}v^*_k\xi_k.$
We simplify the calculation into the following equality
\begin{equation}\label{simple}
{\bf L}(\varphi_{0,1,m_1},\varphi_{n,l,m})=\frac{1}{\sqrt{\mu}}\left[\mathbf{H_1}(v)-\mathbf{H_2}(v)\right].
\end{equation} By using the explicit formula of the Fourier transform of $\Psi_{n,l,m}$
 that,
 $$\widehat{\Psi_{n,l,m}}=B_{n,l}|\xi|^{2n+l}e^{-\frac{|\xi|^2}{2}}Y^{m}_l(\frac{\xi}{|\xi|})$$
 with $B_{n,l}$ defined in \eqref{B_NL}, we have
 $$|\xi|^2\widehat{\Psi_{n,l,m}}(\xi)
 =B_{n,l}|\xi|^{2(n+1)+l}e^{-\frac{|\xi|^2}{2}}Y^{m}_l(\frac{\xi}{|\xi|})=\frac{B_{n,l}}{B_{n+1,l}}\widehat{\Psi_{n+1,l,m}}.$$
It follows from the inverse Fourier transformation of $\widehat{\mathbf{H_1}}(\xi)$ and the equality \eqref{orth-1} in Lemma \ref{addition},
\begin{align*}
\mathbf{H_{1}}(v)&=2\frac{B_{n,l}}{B_{n+1,l}}\sum^3_{j=1}\left(\int_{\mathbb{R}^3}v^*_j\Psi_{0,1,m_1}(v_*)dv_*\right)v_j\Psi_{n+1,l,m}(v)\\
&=2\frac{B_{n,l}}{B_{n+1,l}}\left(\int_{\mathbb{R}^3}(v_*\cdot\,v)\Psi_{0,1,m_1}(v_*)dv_*\right)\Psi_{n+1,l,m}(v)\\
&=2\sqrt{\frac{4\pi}{3}}\frac{B_{n,l}}{B_{n+1,l}}\, |v|\, Y^{m_1}_1\left(\frac{v}{|v|}\right)\Psi_{n+1,l,m}(v)\\
&=2\sqrt{\frac{4\pi}{3}}\frac{(n+1)!B_{n,l}}{(-i)^l2^{(n+1)}}|v|^{l+1}e^{-\frac{|v|^{2}}{2}}
L^{(l+1/2)}_{n+1}\left(\frac{|v|^{2}}{2}\right)\\
&\qquad\times\left(\tilde{C}^{m_1,m}_{l,l+1}
Y^{m_1+m}_{l+1}(\sigma)+\tilde{C}^{m_1,m}_{l,l-1}
Y^{m_1+m}_{l-1}(\sigma)\right)
\end{align*}
Using the formulas $(10')$,$(11)$,$(12)$ of Sec.1, Chap.IV in \cite{San} for $x=\frac{|v|^2}{2}$, we have
$$L^{(l+\frac{1}{2})}_{n+1}(x)
=L^{(l+\frac{3}{2})}_{n+1}(x)
-L^{(l+\frac{3}{2})}_{n}(x),\quad\,n=0,1,\cdots$$
and
\begin{align*}
xL^{(l+\frac{1}{2})}_{n+1}(x)
=(n+l+\frac{3}{2})L^{(l-\frac{1}{2})}_{n+1}(x)
-(n+2)L^{(l-\frac{1}{2})}_{n+2}(x).
\end{align*}
Direct calculation shows that
\begin{align*}
\mathbf{H_{1}}(v)&=4\sqrt{\frac{\pi}{3}}\tilde{C}^{m_1,m}_{l,l+1}\sqrt{2(n+1)(2n+2l+3)(2n+2l+5)}\Psi_{n+1,l+1,m_1+m}(v)\\
&\quad-4\sqrt{\frac{\pi}{3}}\tilde{C}^{m_1,m}_{l,l+1}2(n+1)\sqrt{2n+2l+3}\Psi_{n,l+1,m_1+m}(v)\\
&\quad+4\sqrt{\frac{\pi}{3}}\tilde{C}^{m_1,m}_{l,l-1}(2n+2l+3)\sqrt{2(n+1)}\Psi_{n+1,l-1,m_1+m}(v)\\
&\quad-4\sqrt{\frac{\pi}{3}}\tilde{C}^{m_1,m}_{l,l-1}\sqrt{4(n+1)(n+2)(2n+2l+3)}\Psi_{n+2,l-1,m_1+m}(v).
\end{align*}
Now we calculate $\mathbf{H_{2}}$.
We can deduce from the equality \eqref{orth-1} in Lemma \ref{addition} with $v=\xi$ that
\begin{align*}
\widehat{\mathbf{H_{2}}}(\xi)
&=2i\sqrt{\frac{4\pi}{3}\, }|\xi|\, Y^{m_1}_1(\frac{\xi}{|\xi|})\\
&\qquad\times[(2n+l+4)\widehat{\Psi_{n,l,m}}-\sqrt{2(n+1)(2n+2l+3)}\widehat{\Psi_{n+1,l,m}}]
\end{align*}
By using the Corollary \ref{harmonic-c} and the explicit formula of $\widehat{\Psi_{n,l,m}}$,  one can verify that
\begin{align*}
\widehat{\mathbf{H_{2}}}(\xi)
&=2i\sqrt{\frac{4\pi}{3}}(2n+l+4)B_{n,l}|\xi|^{2n+l+1}e^{-\frac{|\xi|^2}{2}}\\
&\qquad\times\left(\tilde{C}^{m_1,m}_{l,l+1}
Y^{m_1+m}_{l+1}+\tilde{C}^{m_1,m}_{l,l-1}
Y^{m_1+m}_{l-1}\right)\\
&\quad-2i\sqrt{\frac{4\pi}{3}}\sqrt{2(n+1)(2n+2l+3)}B_{n+1,l}|\xi|^{2n+l+3}e^{-\frac{|\xi|^2}{2}}\\
&\qquad\times\left(\tilde{C}^{m_1,m}_{l,l+1}
Y^{m_1+m}_{l+1}+\tilde{C}^{m_1,m}_{l,l-1}
Y^{m_1+m}_{l-1}\right)\\
&=-4\sqrt{\frac{\pi}{3}}\tilde{C}^{m_1,m}_{l,l+1}(2n+l+4)\sqrt{2n+2l+3}\,\mathcal{F}(\Psi_{n,l+1,m_1+m})\\
&\quad+4\sqrt{\frac{\pi}{3}}\tilde{C}^{m_1,m}_{l,l-1}(2n+l+4)\sqrt{2(n+1)}\,\mathcal{F}(\Psi_{n+1,l-1,m_1+m})\\
&\quad+4\sqrt{\frac{\pi}{3}}\tilde{C}^{m_1,m}_{l,l+1}\sqrt{2(n+1)(2n+2l+3)}\sqrt{2n+2l+5}\,\mathcal{F}(\Psi_{n+1,l+1,m_1+m})\\
&\quad-4\sqrt{\frac{\pi}{3}}\tilde{C}^{m_1,m}_{l,l-1}\sqrt{4(n+1)(n+2)(2n+2l+3)}\,\mathcal{F}(\Psi_{n+2,l-1,m_1+m}).
\end{align*}
Then by the inverse Fourier transform of $\widehat{\mathbf{H_{2}}}(\xi)$ and substituting the equalities of $\mathbf{H_1}, \mathbf{H_2}$ into \eqref{simple}, we conclude that
\begin{align*}
&{\bf L}(\varphi_{0,1,m_1},\varphi_{n,l,m})=\frac{1}{\sqrt{\mu}}\left[\mathbf{H_1}(v)-\mathbf{H_2}(v)\right]\\
&=A^-_{n,l,m,m_1}\varphi_{n+1,l-1,m_1+m}+A^+_{n,l,m,m_1}\varphi_{n,l+1,m_1+m}.
\end{align*}
\end{proof}

\begin{lemma}\label{lemma5.3+}
For $n,l\in\mathbb{N}$, $|m|\leq\,l$, we have
\begin{align*}
{\bf L}(\varphi_{1,0,0},\varphi_{n,l,m})=\frac{4\sqrt{3(n+1)(2n+2l+3)}}{3}
\varphi_{n+1,l,m}.
\end{align*}
\end{lemma}

\begin{proof} Firstly
$$
{\bf L}(\varphi_{1,0,0},\varphi_{n,l,m})={\bf L}_1(\varphi_{1,0,0},\varphi_{n,l,m})-{\bf L}_2(\varphi_{1,0,0},\varphi_{n,l,m}).
$$
Using
\begin{align*}
&\int_{\mathbb{R}^{3}}
\partial_{v^*_k}\Psi_{1,0,0}(v_*)dv_*
=\left.i\xi_k\widehat{\Psi_{1,0,0}}(\xi)\right|_{\xi=0}=0,\quad\,k=1,2,3;\\
&\int_{\mathbb{R}^{3}}
v^*_j\partial_{v^*_k}\Psi_{1,0,0}(v_*)dv_*
=-\left.\xi_k\partial_{\xi^*_j}\widehat{\Psi_{1,0,0}}(\xi)\right|_{\xi=0}=0,\quad\,k\neq\,j;\\
&\int_{\mathbb{R}^{3}}
(v^*_j)^2\partial_{v^*_k}\Psi_{1,0,0}(v_*)dv_*
=-\left.i\xi_k\partial^2_{\xi^*_j}\widehat{\Psi_{1,0,0}}(\xi)\right|_{\xi=0}=0,\quad\,k\neq\,j;\\
&\int_{\mathbb{R}^{3}}
v^*_j\partial_{v^*_j}\Psi_{1,0,0}(v_*)dv_*
=\left.-\widehat{\Psi_{1,0,0}}(0)
-\xi_j\partial_{\xi^*_j}\widehat{\Psi_{1,0,0}}(\xi)\right|_{\xi=0}=0;\\
&\int_{\mathbb{R}^{3}}
v^*_kv^*_j\partial_{v^*_j}\Psi_{1,0,0}(v_*)dv_*
=-\int_{\mathbb{R}^{3}}
v^*_k\Psi_{1,0,0}(v_*)dv_*,
\end{align*}
and
\begin{align*}
&\int_{\mathbb{R}^{3}}\Psi_{1,0,0}(v_*)dv_*
=(\varphi_{1,0,0},\varphi_{0,0,0})_{L^2(\mathbb{R}^3)}=0,\\
&\int_{\mathbb{R}^{3}}v^*_k\Psi_{1,0,0}(v_*)dv_*=0,\quad\forall\,1\leq\,k\leq3.
\end{align*}
We can conclude that
\begin{align*}
&{\bf L}(\varphi_{1,0,0},\varphi_{n,l,m})
={\bf L}_1(\varphi_{1,0,0},\varphi_{n,l,m})-{\bf L}_2(\varphi_{1,0,0},\varphi_{n,l,m})\\
&=\frac{1}{\sqrt{\mu(v)}}\sum_{\substack{1\leq\,k,j\leq3\\k\neq\,j}}
\left(\int_{\mathbb{R}^3}|v^*_j|^2\Psi_{1,0,0}(v_*)dv_*\right)\partial^2_{v_k}\Psi_{n,l,m}(v)\\
&\quad+\frac{1}{\sqrt{\mu(v)}}\sum_{\substack{1\leq\,k,j\leq3\\k\neq\,j}}\left(\int_{\mathbb{R}^3}v^*_k v^*_j\Psi_{1,0,0}(v_*)dv_*\right)\partial_{v_k}\partial_{v_j}\Psi_{n,l,m}(v).
\end{align*}
Again by the Fourier transformation of $\sqrt{\mu}\, {\bf L}(\varphi_{1,0,0},\varphi_{n,l,m})$, we have
\begin{align*}
&\mathcal{F}[\sqrt{\mu}\, {\bf L}(\varphi_{1,0,0},\varphi_{n,l,m})](\xi)\\
&=\int_{\mathbb{R}_{v_*}^3}\left(\xi\cdot\,v_*\right)^2\Psi_{1,0,0}(v_*)dv_*\widehat{\Psi_{n,l,m}}(\xi)
-\left(|v_*|^2\sqrt{\mu},\,\varphi_{1,0,0}\right)_{L^2(\mathbb{R}^3)}|\xi|^2\widehat{\Psi_{n,l,m}}(\xi).
\end{align*}
We deduce from  the equality \eqref{orth-3} in Lemma \ref{addition} that
\begin{equation*}
\int_{\mathbb{R}^3_{v_*}}(\xi\cdot\,v_*)^2\Psi_{1,0,0}(v_*)dv_*
=-\frac{\sqrt{6}}{3}|\xi|^2.
\end{equation*}
From the definition of eigenfunctions $\varphi_{1,0,0}, \varphi_{0,0,0}$, one can calculate that
 $$
|v|^2\sqrt{\mu}=3\varphi_{0,0,0}-\sqrt{6}\varphi_{1,0,0}
$$
Then
$$
-\left(|v_*|^2\sqrt{\mu},\,\varphi_{1,0,0}\right)_{L^2(\mathbb{R}^3)}=\sqrt{6}.
$$
This implies that
\begin{align*}
&\mathcal{F}[\sqrt{\mu}\, {\bf L}(\varphi_{1,0,0},\varphi_{n,l,m})](\xi)=\frac{2\sqrt{6}}{3}|\xi|^2\widehat{\Psi_{n,l,m}}(\xi)\\
&=\frac{4\sqrt{3(n+1)(2n+2l+3)}}{3}\widehat{\Psi_{n+1,l,m}}(\xi).
\end{align*}
We end the proof of the Lemma by inverse Fourier transformation.
\end{proof}

\begin{lemma}\label{lemma5.4+}
For $n,l\in\mathbb{N}$, $|m|\leq\,l$, we have
\begin{align*}
&{\bf L}(\varphi_{0,2,m_2},\varphi_{n,l,m})
=A^1_{n,l,m,m_2}
\varphi_{n+2,l-2,m+m_2}\\
&\qquad\quad+A^2_{n,l,m,m_2}\varphi_{n+1,l,m+m_2}+A^3_{n,l,m,m_2}\varphi_{n,l+2,m+m_2}, \,\forall\,|m_2|\leq 2
\end{align*}
with
\begin{equation}\label{An}
\begin{split}
&A^1_{n,l,m,m_2}=-4\sqrt{\frac{\pi}{15}}
\sqrt{4(n+2)(n+1)}\int_{\mathbb{S}^2}Y^{m_2}_2(\omega)
Y^{m}_{l}(\omega)Y^{-m_2-m}_{l-2}(\omega)d\omega;\\
&A^2_{n,l,m,m_2}=4\sqrt{\frac{\pi}{15}}
\sqrt{2(n+1)(2n+2l+3)}\\
&\qquad\qquad\qquad\qquad\times\int_{\mathbb{S}^2}Y^{m_2}_2(\omega)
Y^{m}_{l}(\omega)Y^{-m_2-m}_{l}(\omega)d\omega;\\
&A^3_{n,l,m,m_2}=-4\sqrt{\frac{\pi}{15}}
\sqrt{(2n+2l+5)(2n+2l+3)}\\
&\qquad\qquad\qquad\qquad\times\int_{\mathbb{S}^2}Y^{m_2}_2(\omega)
Y^{m}_{l}(\omega)Y^{-m_2-m}_{l+2}(\omega)d\omega.
\end{split}
\end{equation}
\end{lemma}

\begin{proof}
For $|m_2|\leq2$, for all $n,l\in\mathbb{N},|m|\le\,l$, one has
$$
{\bf L}(\varphi_{0,2,m_2},\varphi_{n,l,m})={\bf L}_1(\varphi_{0,2,m_2},\varphi_{n,l,m})-{\bf L}_2(\varphi_{0,2,m_2},
\varphi_{n,l,m}).
$$
Using now
\begin{align*}
&\int_{\mathbb{R}^{3}}
\partial_{v^*_k}\Psi_{0,2,m_2}(v_*)dv_*
=\left.i\xi_k\widehat{\Psi_{0,2,m_2}}(\xi)\right|_{\xi=0}=0,\quad\,k=1,2,3;\\
&\int_{\mathbb{R}^{3}}
v^*_j\partial_{v^*_k}\Psi_{0,2,m_2}(v_*)dv_*
=-\left.\xi_k\partial_{\xi^*_j}\widehat{\Psi_{0,2,m_2}}(\xi)\right|_{\xi=0}=0,\quad\,k\neq\,j;\\
&\int_{\mathbb{R}^{3}}
(v^*_j)^2\partial_{v^*_k}\Psi_{0,2,m_2}(v_*)dv_*
=-\left.i\xi_k\partial^2_{\xi^*_j}\widehat{\Psi_{0,2,m_2}}(\xi)\right|_{\xi=0}=0,\quad\,k\neq\,j;\\
&\int_{\mathbb{R}^{3}}
v^*_j\partial_{v^*_j}\Psi_{0,2,m_2}(v_*)dv_*
=\left.-\widehat{\Psi_{0,2,m_2}}(0)
-\xi_j\partial_{\xi^*_j}\widehat{\Psi_{0,2,m_2}}(\xi)\right|_{\xi=0}=0;\\
&\int_{\mathbb{R}^{3}}
v^*_kv^*_j\partial_{v^*_j}\Psi_{0,2,m_2}(v_*)dv_*
=-\int_{\mathbb{R}^{3}}
v^*_k\Psi_{0,2,m_2}(v_*)dv_*,
\end{align*}
and
\begin{align*}
&\int_{\mathbb{R}^{3}}\Psi_{0,2,m_2}(v_*)dv_*
=(\varphi_{0,2,m_2},\varphi_{0,0,0})_{L^2(\mathbb{R}^3)}=0,\\
&\int_{\mathbb{R}^{3}}v^*_k\Psi_{0,2,m_2}(v_*)dv_*=0,\quad\forall\,1\leq\,k\leq3.
\end{align*}
We can conclude that
\begin{align*}
&{\bf L}(\varphi_{0,2,m_2},\varphi_{n,l,m})
={\bf L}_1(\varphi_{0,2,m_2},\varphi_{n,l,m})-{\bf L}_2(\varphi_{0,2,m_2},\varphi_{n,l,m})\\
&=\frac{1}{\sqrt{\mu(v)}}\sum_{\substack{1\leq\,k,j\leq3\\k\neq\,j}}
\left(\int_{\mathbb{R}^3}|v^*_j|^2\Psi_{0,2,m_2}(v_*)dv_*\right)\partial^2_{v_k}\Psi_{n,l,m}(v)\\
&\quad+\frac{1}{\sqrt{\mu(v)}}\sum_{\substack{1\leq\,k,j\leq3\\k\neq\,j}}\left(\int_{\mathbb{R}^3}v^*_k v^*_j\Psi_{0,2,m_2}(v_*)dv_*\right)\partial_{v_k}\partial_{v_j}\Psi_{n,l,m}(v).
\end{align*}
By the Fourier transformation, we have
\begin{align*}
&\mathcal{F}[\sqrt{\mu}\, {\bf L}(\varphi_{0,2,m_2},\varphi_{n,l,m})](\xi)\\
&=-\left(|v_*|^2\sqrt{\mu},\,\varphi_{0,2,m_2}\right)_{L^2(\mathbb{R}^3)}|\xi|^2\widehat{\Psi_{n,l,m}}(\xi)\\
&\quad+\int_{\mathbb{R}_{v_*}^3}\left(\xi\cdot\,v_*\right)^2\Psi_{0,2,m_2}(v_*)
dv_*\widehat{\Psi_{n,l,m}}(\xi).
\end{align*}
By using
$$
\left(|v_*|^2\sqrt{\mu},\,\varphi_{0,2,m_2}\right)_{L^2(\mathbb{R}^3)}=\left(3\varphi_{0,0,0}-\sqrt{6}\varphi_{1,0,0},\varphi_{0,2,m_2}\right)_{L^2(\mathbb{R}^3)}=0.
$$
and by using the equality \eqref{orth-2} with $v=\xi$ in Lemma \ref{addition} that
\begin{equation*}
\int_{\mathbb{R}^3_{v_*}}(\xi\cdot\,v_*)^2\Psi_{0,2,m_2}(v_*)dv_*=\sqrt{\frac{16\pi}{15}}|\xi|^2Y^{m_2}_2(\frac{\xi}{|\xi|}),
\end{equation*}
we obtain
\begin{align*}
\mathcal{F}[\sqrt{\mu}\, {\bf L}(\varphi_{0,2,m_2},\varphi_{n,l,m})](\xi)=4\sqrt{\frac{\pi}{15}}|\xi|^2Y^{m_2}_2\left(\frac{\xi}{|\xi|}\right)\widehat{\Psi_{n,l,m}}(\xi).
\end{align*}
We apply Corollary \ref{harmonic-c} with $l=2$, $\omega=\frac{\xi}{|\xi|}$ that
$$Y^{m_2}_2(\omega)Y^{m}_{l}(\omega)=
C^{m_2,m}_{l,l-2}Y^{m_2+m}_{l-2}(\omega)
+C^{m_2,m}_{l,l}Y^{m_2+m}_{l}(\omega)
+C^{m_2,m}_{l,l+2}Y^{m_2+m}_{l+2}(\omega).$$
Recalled  that
$$\widehat{\Psi_{n,l,m}}(\xi)=B_{n,l}|\xi|^{2n+l}e^{-\frac{|\xi|^2}{2}}Y^m_l(\frac{\xi}{|\xi|}),$$
we have
\begin{align*}
&\mathcal{F}[\sqrt{\mu}\, {\bf L}(\varphi_{0,2,m_2},\varphi_{n,l,m})](\xi)\\
&=4\sqrt{\frac{\pi}{15}}B_{n,l}|\xi|^{2n+l+2}e^{-\frac{|\xi|^2}{2}}\\
&\quad\times\left(C^{m_2,m}_{l,l-2}Y^{m_2+m}_{l-2}(\omega)
+C^{m_2,m}_{l,l}Y^{m_2+m}_{l}(\omega)
+C^{m_2,m}_{l,l+2}Y^{m_2+m}_{l+2}(\omega)\right)\\
&=4\sqrt{\frac{\pi}{15}}\frac{B_{n,l}}{B_{n+2,l-2}}C^{m_2,m}_{l,l-2}\widehat{\Psi_{n+2,l-2,m}}(\xi)+4\sqrt{\frac{\pi}{15}}\frac{B_{n,l}}{B_{n+1,l}}C^{m_2,m}_{l,l-2}\widehat{\Psi_{n+1,l,m}}(\xi)\\
&\qquad+4\sqrt{\frac{\pi}{15}}\frac{B_{n,l}}{B_{n,l+2}}C^{m_2,m}_{l,l-2}\widehat{\Psi_{n,l+2,m}}(\xi).
\end{align*}
Direct calculation and the inverse Fourier transform implies that
\begin{align*}
{\bf L}(\varphi_{0,2,m_2},\varphi_{n,l,m})
&=A^1_{n,l,m,m_2}
\varphi_{n+2,l-2,m+m_2}\\
&\quad+A^2_{n,l,m,m_2}\varphi_{n+1,l,m+m_2}+A^3_{n,l,m,m_2}\varphi_{n,l+2,m+m_2}\,.
\end{align*}
This ends the proof of Lemma.
\end{proof}

\begin{lemma}\label{lemma5.5+}
For $n,l\in\mathbb{N}$, $|m|\leq\,l$, we have
\begin{align*}
{\bf L}(\varphi_{\tilde n,\tilde{l},\tilde{m}},\varphi_{n,l,m})=0, \,\,\,\,\forall\,2\tilde n+\tilde{l}>2,\, |\tilde m|\le \tilde l.
\end{align*}
\end{lemma}

\begin{proof}
For any $n,l\in\mathbb{N}, \tilde{m}\in\mathbb{Z}$,  and $2\tilde{n}+\tilde{l}>2$, $|\tilde{m}|\leq \tilde{l}$,  we have again
$$
{\bf L}(\varphi_{\tilde{n},\tilde{l},\tilde{m}},\varphi_{n,l,m})={\bf L}_1(\varphi_{\tilde{n},\tilde{l},\tilde{m}},\varphi_{n,l,m})-{\bf L}_2(\varphi_{\tilde{n},\tilde{l},\tilde{m}},\varphi_{n,l,m}).
$$
Using the facts
\begin{align*}
&\int_{\mathbb{R}^{3}}
\partial_{v^*_k}\Psi_{\tilde{n},\tilde{l},\tilde{m}}(v_*)dv_*
=\left.i\xi_k\widehat{\Psi_{\tilde{n},\tilde{l},\tilde{m}}}(\xi)\right|_{\xi=0}=0,\quad\,k=1,2,3;\\
&\int_{\mathbb{R}^{3}}
v^*_j\partial_{v^*_k}\Psi_{\tilde{n},\tilde{l},\tilde{m}}(v_*)dv_*
=-\left.\xi_k\partial_{\xi^*_j}\widehat{\Psi_{\tilde{n},\tilde{l},\tilde{m}}}(\xi)\right|_{\xi=0}=0,\quad\,k\neq\,j;\\
&\int_{\mathbb{R}^{3}}
(v^*_j)^2\partial_{v^*_k}\Psi_{\tilde{n},\tilde{l},\tilde{m}}(v_*)dv_*
=-\left.i\xi_k\partial^2_{\xi^*_j}\widehat{\Psi_{\tilde{n},\tilde{l},\tilde{m}}}(\xi)\right|_{\xi=0}=0,\quad\,k\neq\,j;\\
&\int_{\mathbb{R}^{3}}
v^*_j\partial_{v^*_j}\Psi_{\tilde{n},\tilde{l},\tilde{m}}(v_*)dv_*
=\left.-\widehat{\Psi_{\tilde{n},\tilde{l},\tilde{m}}}(0)
-\xi_j\partial_{\xi^*_j}\widehat{\Psi_{\tilde{n},\tilde{l},\tilde{m}}}(\xi)\right|_{\xi=0}=0;\\
&\int_{\mathbb{R}^{3}}
v^*_kv^*_j\partial_{v^*_j}\Psi_{\tilde{n},\tilde{l},\tilde{m}}(v_*)dv_*
=-\int_{\mathbb{R}^{3}}
v^*_k\Psi_{\tilde{n},\tilde{l},\tilde{m}}(v_*)dv_*.
\end{align*}
By the relation \eqref{relation2}, one can verify that, for any $2\tilde{n}+\tilde{l}>2$, $|\tilde{m}|\leq \tilde{l}$,
\begin{align*}
&\int_{\mathbb{R}^{3}}\Psi_{\tilde{n},\tilde{l},\tilde{m}}(v_*)dv_*
=(\varphi_{\tilde{n},\tilde{l},\tilde{m}},\varphi_{0,0,0})_{L^2(\mathbb{R}^3)}=0,\\
&\int_{\mathbb{R}^{3}}v^*_k\Psi_{\tilde{n},\tilde{l},\tilde{m}}(v_*)dv_*=0,\quad\forall\,1\leq\,k\leq3\\
&\int_{\mathbb{R}^{3}}v^*_kv^*_j\Psi_{\tilde{n},\tilde{l},\tilde{m}}(v_*)dv_*=0,\quad\forall\,1\leq\,k,j\leq3.
\end{align*}
We can conclude that, for all $2\tilde{n}+\tilde{l}>2$, $|\tilde{m}|\leq \tilde{l}$,
\begin{align*}
{\bf L}(\varphi_{\tilde{n},\tilde{l},\tilde{m}},\varphi_{n,l,m})\equiv0.
\end{align*}
\end{proof}

For the proof of the Proposition \ref{recur}, we recall the elementary result about the Legendre polynomial in the following.
\begin{lemma}\label{recurrence}
Let $l\in\mathbb{N}$ be nonnegative integer, $P_l(x)$ is the Legendre polynomial,  we have
\begin{equation*}
\left\{ \begin{aligned}
P_2(x)P_0(x)&=P_2(x),\\
P_2(x)P_1(x)&=\frac{3}{5}P_3(x)+\frac{2}{5}P_1(x),\\
P_1(x)P_l(x)&=\frac{l+1}{2l+1}P_{l+1}(x)+\frac{l}{2l+1}P_{l-1}(x),\\
P_2(x)P_l(x)&=\frac{3(l+2)(l+1)}{2(2l+3)(2l+1)}P_{l+2}(x)+\frac{(l+1)l}{(2l+3)(2l-1)}P_l(x)\\
         &\qquad+\frac{3l(l-1)}{2(2l+1)(2l-1)}P_{l-2}(x)\quad \text{for}\quad\,l\ge2.
\end{aligned} \right.
\end{equation*}
\end{lemma}

For more general case, we can refer to the Example 11 in Chap.XV in \cite{Whit-Watson} or $(1.4)$ in Appendix 1 in \cite{Jones}.

\bigskip
\noindent
{\bf Proof of the Proposition \ref{recur}}

Recalled from Lemma \ref{lemma5.4+} that
\begin{align*}
A^1_{n-2,l+2,m,m_2}&=-4\sqrt{\frac{\pi}{15}}
\sqrt{4n(n-1)}\int_{\mathbb{S}^2}Y^{m_2}_2(\omega)
Y^{m}_{l+2}(\omega)Y^{-m_2-m}_{l}(\omega)d\omega;\\
A^2_{n-1,l,m,m_2}=&4\sqrt{\frac{\pi}{15}}
\sqrt{2n(2n+2l+1)}\int_{\mathbb{S}^2}Y^{m_2}_2(\omega)
Y^{m}_{l}(\omega)Y^{-m_2-m}_{l}(\omega)d\omega;\\
A^3_{n,l-2,m,m_2}=&-4\sqrt{\frac{\pi}{15}}
\sqrt{(2n+2l+1)(2n+2l-1)}\int_{\mathbb{S}^2}Y^{m_2}_2
Y^{m}_{l-2}Y^{-m_2-m}_{l}d\omega.
\end{align*}
We recalled the addition theorem (7-34) of Chapter 7 in \cite{JCSlater}, (VIII) of Sec.19, Chap. III in \cite{San} or Theorem 1 of Sec.4, Chap. 1 in \cite{CM} that, for $\sigma,\kappa\in\mathbb{S}^2$,
$$P_k(\sigma\cdot\kappa)=\frac{4\pi}{2k+1}\sum_{|m|\leq\,k}Y^m_k(\sigma)Y^{-m}_k(\kappa),\quad\forall\,k\in\mathbb{N}.$$
Therefore, for any $m^*\in\mathbb{Z}$ and $|m^*|\leq\,l$,
\begin{align*}
&\sum_{\substack{|m|\leq\,l+2, |m_2|\leq2\\m+m_2=m^*}}\left|A^1_{n-2,l+2,m,m_2}\right|^2\\
&=\frac{64n(n-1)\pi}{15}\sum_{|m|\leq\,l+2}\sum_{|m_2|\leq2}\left|\int_{S^2}Y^{m}_{l+2}Y_2^{m_2}Y^{m^*}_ld\sigma\right|^2\\
&=\frac{64n(n-1)\pi}{15}\int_{S^2_{\sigma}}\int_{S^2_{\kappa}}\frac{5}{4\pi}\frac{2l+5}{4\pi}P_2(\kappa\cdot\sigma)P_{l+2}(\kappa\cdot\sigma)
Y^{m^*}_l(\kappa)Y^{-m^*}_l(\sigma)d\kappa\,d\sigma.
\end{align*}
By using Lemma \ref{recurrence} and the orthogonal of $\{Y^m_l\}_{l\in\mathbb{N}, |m|\leq\,l}$ on $\mathbb{S}^2$, we have, for $n\ge2$,
\begin{align*}
&\sum_{\substack{|m|\leq\,l+2, |m_2|\leq2\\m+m_2=m^*}}\left|A^1_{n-2,l+2,m,m_2}\right|^2\\
&=\frac{64n(n-1)\pi}{15}\frac{5}{4\pi}\frac{2l+5}{4\pi}\frac{3(l+2)(l+1)}{2(2l+5)(2l+3)}\frac{4\pi}{2l+1}
\\
&=\frac{8n(n-1)(l+2)(l+1)}{(2l+3)(2l+1)}\leq\frac{16n(n-1)}{3}.
\end{align*}
This is the estimation \eqref{A1}.  Similar to the proof of \eqref{A1},
one can deduce also from Lemma \ref{recurrence} and the orthogonal of $\{Y^m_l\}_{l\in\mathbb{N}, |m|\leq\,l}$ on $\mathbb{S}^2$ that
\begin{align*}
&\Big|A^2_{n-1,0,0,0}\Big|^2\\
&=\frac{32n(2n+1)\pi}{15}\int_{S^2_{\sigma}}\int_{S^2_{\kappa}}\frac{5}{4\pi}\frac{1}{4\pi}P_2(\kappa\cdot\sigma)P_0(\kappa\cdot\sigma)
Y^{0}_0(\kappa)Y^{0}_0(\sigma)d\kappa\,d\sigma\\
&=\frac{32n(2n+1)\pi}{15}\frac{5}{4\pi}\frac{1}{4\pi}\int_{S^2_{\sigma}}\int_{S^2_{\kappa}}P_2(\kappa\cdot\sigma)
Y^{0}_0(\kappa)Y^{0}_0(\sigma)d\kappa\,d\sigma
=0,\quad\forall\,n\ge1,
\end{align*}
and for any $m^*\in\mathbb{Z}$ and $|m^*|\leq\,l$
\begin{align*}
\sum_{\substack{|m|\leq\,l, |m_2|\leq2\\m+m_2=m^*}}\Big|A^2_{n-1,l,m,m_2}\Big|^2
&\leq\frac{32n(2n+2l+1)\pi}{15}\frac{5}{4\pi}\frac{(l+1)l}{(2l+3)(2l-1)}\\
&\leq\frac{4 n(2n+2l+1)}{3}, \quad\forall\,n\ge1,l\ge1.
\end{align*}
Finally, one can estimate that
\begin{align*}
&\sum_{\substack{|m|\leq\,l-2, |m_2|\leq2\\m+m_2=m^*}}\left|A^3_{n,l-2,m,m_2}\right|^2\\
&=\frac{16(2n+2l+1)(2n+2l-1)\pi}{15}\frac{5}{4\pi}\frac{3l(l-1)}{2(2l+1)(2l-1)}\\
&\leq\frac{(2n+2l+1)(2n+2l-1)}{2} \quad\forall\,n\in\mathbb{N},\,l\ge2.
\end{align*}
The estimations \eqref{A2} and \eqref{A3} follow.  We end the proof of Proposition \ref{recur}.


\section{Appendix}\label{appendix}

The proof of the example and the characterization of the Gelfand-Shilov spaces and the Shubin spaces are presented in this section. For the self-content of paper,
we will present some proof here.

\begin{proof}[The proof of the Example \ref{Example}] Now we prove that the function $g_0$ defined in \eqref{ex1} is belongs to $ Q^{\alpha}(\mathbb{R}^3)\cap \mathcal{N}^{\perp}$ and $\|\mathbb{S}_2g\|_{L^2(\mathbb{R}^3)}=0$ for $\alpha<-\frac{3}{2}$.

Recalled the spectrum functions $\varphi_{n,l,m}(v)$ with $2n+l\leq2,\,|m|\leq\,l$, we have
$$
g_0=\frac{1}{\sqrt{\mu}}\delta_{\mathbf{0}}-
\left(\frac{5}{2}-\frac{|v|^2}{2}\right)\sqrt{\mu}
=\frac{1}{\sqrt{\mu}}\delta_{\mathbf{0}}-\varphi_{0,0,0}-\sqrt{\frac{3}{2}}\varphi_{1,0,0}.
$$
One can calculate directly that
\begin{align*}
&\langle\,g_0, \varphi_{0,0,0}\rangle=\langle\delta_0, 1\rangle-\langle\varphi_{0,0,0},\varphi_{0,0,0}\rangle=0;\\
&\langle\,g_0, \varphi_{1,0,0}\rangle=\langle \delta_0,\, \sqrt{\frac{2}{3}}\left(\frac{3}{2}-\frac{|v|^2}{2}\right)\rangle-\sqrt{\frac{3}{2}}\langle\varphi_{1,0,0},\varphi_{1,0,0}\rangle=0.
\end{align*}
Since $g_0$ is radial, we can verify that
$$
\langle g_0, \varphi_{n,l,m}\rangle\equiv0,\quad\forall\,2n+l\leq2,\,|m|\leq\,l.
$$
This shows that $g_0\in\mathcal{N}^{\perp}$ and $\|\mathbb{S}_2g\|_{L^2(\mathbb{R}^3)}=0$.  Now we prove that $g_0\in \,Q^{\alpha}(\mathbb{R}^3)$ for $\alpha<-\frac{3}{2}$. Since $g_0\in\mathcal{N}^{\perp}$ and radial, we can write $g_0$ in the form
$$
g_0=\sum^{+\infty}_{k=2}\langle g_0, \varphi_{k,0,0}\rangle\, \varphi_{k,0,0},
$$
where we can calculate in details that,
$$
\langle g_0, \varphi_{k,0,0}\rangle=\langle \mu^{-\frac 12}\delta_{\mathbf{0}}, \varphi_{k,0,0}\rangle=\sqrt{\frac{2\Gamma(k+\frac{3}{2})}{\sqrt{\pi}k!}}.
$$
By using the Stirling equivalent
$$
\Gamma(x+1)\sim_{x\rightarrow+\infty}\sqrt{2\pi x}\left(\frac{x}{e}\right)^x,
$$
we have that, $\forall k\geq2$
$$
\sqrt{\frac{2\Gamma(k+\frac{3}{2})}{\sqrt{\pi}k!}}\sim k^{\frac{1}{4}}\, .
$$
Therefore, for any $\alpha<-\frac{3}{2}$,
$$
\|g_0\|^2_{Q^{\alpha}(\mathbb{R}^3)}
=\sum^{+\infty}_{k=2}(4k)^{\alpha}|\langle g_0, \varphi_{k,0,0}\rangle|^2
\lesssim\sum^{+\infty}_{k=2}k^{\alpha+\frac{1}{2}}<+\infty.
$$
This implies that $g_0\in\,Q^{\alpha}(\mathbb{R}^3)$, we end the proof of the Example.
\end{proof}

\smallskip\noindent
{\bf Gelfand-Shilov spaces.}  The symmetric Gelfand-Shilov space $S^{\nu}_{\nu}(\mathbb{R}^3)$ can be characterized through the decomposition
into the Hermite basis $\{H_{\alpha}\}_{\alpha\in\mathbb{N}^3}$ and the harmonic oscillator $\mathcal{H}=-\triangle +\frac{|v|^2}{4}$.
For more details, see Theorem 2.1 in \cite{GPR}
\begin{align*}
f\in S^{\nu}_{\nu}(\mathbb{R}^3)&\Leftrightarrow\,f\in C^\infty (\mathbb{R}^3),\exists\, \tau>0, \|e^{\tau\mathcal{H}^{\frac{1}{2\nu}}}f\|_{L^2}<+\infty;\\
&\Leftrightarrow\, f\in\,L^2(\mathbb{R}^3),\exists\, \epsilon_0>0,\,\,\Big\|\Big(e^{\epsilon_0|\alpha|^{\frac{1}{2\nu}}}(f,\,H_{\alpha})_{L^2}\Big)_{\alpha\in\mathbb{N}^3}\Big\|_{l^2}<+\infty;\\
&\Leftrightarrow\,\exists\,C>0,\,A>0,\,\,\|(-\triangle +\frac{|v|^2}{4})^{\frac{k}{2}}f\|_{L^2(\mathbb{R}^3)}\leq AC^k(k!)^{\nu},\,\,\,k\in\mathbb{N}
\end{align*}
where
$$H_{\alpha}(v)=H_{\alpha_1}(v_1)H_{\alpha_2}(v_2)H_{\alpha_3}(v_3),\,\,\alpha\in\mathbb{N}^3,$$
and for $x\in\mathbb{R}$,
$$H_{n}(x)=\frac{(-1)^n}{\sqrt{2^nn!\pi}}e^{\frac{x^2}{2}}\frac{d^n}{dx^n}(e^{-x^2})
=\frac{1}{\sqrt{2^nn!\pi}}\Big(x-\frac{d}{dx}\Big)^n(e^{-\frac{x^2}{2}}).$$
For the harmonic oscillator $\mathcal{H}=-\triangle +\frac{|v|^2}{4}$ of 3-dimension and $s>0$, we have
$$
\mathcal{H}^{\frac{k}{2}} H_{\alpha} = (\lambda_{\alpha})^{\frac{k}{2}}H_{\alpha},\,\, \lambda_{\alpha}=\sum^3_{j=1}(\alpha_j+\frac{1}{2}),\,\,k\in\mathbb{N},\,\alpha\in\mathbb{N}^3.
$$

\smallskip\noindent
{\bf Shubin spaces.} We refer the reader to the works \cite{GPR, Shubin} for the Shubin spaces.
Let $\tau\in\mathbb{R}$, The Shubin spaces $Q^{\tau}(\mathbb{R}^3)$
can be also characterized through the decomposition into the Hermite basis :
\begin{align*}
f\in Q^{\tau}(\mathbb{R}^3)
&\Leftrightarrow\,f\in\,\mathcal{S}'(\mathbb{R}^3),\,\,
\Bigl\|\mathcal{H}^{\frac{\tau}{2}} \, f\Bigr\|_{L^2}<+\infty;
\\
&\Leftrightarrow\, f\in\,\mathcal{S}'(\mathbb{R}^3),\,\,
\Big\|\Big(
(|\alpha|+\frac{3}{2})^{\tau/2} (f,\,H_{\alpha})_{L^2}
\Big)_{\alpha\in\mathbb{N}^3}\Big\|_{l^2}<+\infty\,,
\end{align*}
and for $\tau>0$,
$$
Q^{\tau}(\mathbb{R}^3)\subsetneq H^{\tau}(\mathbb{R}^3)
$$
where $H^{\tau}(\mathbb{R}^3)$ is the usuel Sobolev space. In fact,
$$
\mathcal{H} f\in L^2(\mathbb{R}^3)\,\,\Rightarrow\,\, \triangle f,\,  |v|^2 f\,  \in L^2(\mathbb{R}^3).
$$
So that for the negative index, we have,
$$
H^{-\tau}(\mathbb{R}^3) \subsetneq Q^{-\tau}(\mathbb{R}^3).
$$
See more details in the Appendix in \cite{MPX}.

\bigskip
\noindent {\bf Acknowledgements.}
The first author is supported by the Natural Science Foundation of China under Grant No.11701578. The research of the second author was supported partially by  ``The Fundamental Research Funds for Central Universities of China''.

\end{document}